\documentclass{svjour3}                     % onecolumn (standard format)
\smartqed  % flush right qed marks, e.g. at end of proof
\usepackage{setspace}
\usepackage{algorithm}
\usepackage{algorithmic}

\usepackage{amssymb}
\usepackage{booktabs}
\usepackage{graphicx,fancyhdr}
\usepackage{graphics,color}
\usepackage{subfigure}%并列的子图形
\usepackage{booktabs}
\usepackage{graphicx}
\usepackage{amsmath}
\numberwithin{equation}{section}
\numberwithin{theorem}{section}
\numberwithin{lemma}{section}
\numberwithin{remark}{section}

\usepackage{mathptmx}
\begin{document}

\title{Algorithm implementation and numerical analysis for the two-dimensional tempered fractional Laplacian}

%\titlerunning{Short form of title}        % if too long for running head

\author{Jing Sun$^{1}$, Daxin Nie$^{1}$, Weihua Deng$^{*,1}$
}

%\authorrunning{Short form of author list} % if too long for running head

\institute{
$^{*}$Corresponding author. E-mail: dengwh@lzu.edu.cn            \\
$^{1}$School of Mathematics and Statistics, Gansu Key Laboratory of Applied Mathematics and Complex Systems, Lanzhou University, Lanzhou 730000, P.R. China \\
}
%E-mail: Eli.Barkai@biu.ac.il}

\date{Received: date / Accepted: date}
% The correct dates will be entered by the editor

\maketitle

\begin{abstract}
Tempered fractional Laplacian is the generator of the tempered isotropic L\'evy process [W.H. Deng, B.Y. Li, W.Y. Tian, and P.W. Zhang, Multiscale Model. Simul., 16(1), 125-149, 2018]. This paper provides the finite difference discretization for the two dimensional tempered fractional Laplacian $(\Delta+\lambda)^{\frac{\beta}{2}}$. Then we use it to solve the tempered fractional Poisson equation with Dirichlet boundary conditions and derive the error estimates. Numerical experiments verify the convergence rates and effectiveness of the schemes.

%In this paper, we provide a finite difference scheme by discretizing the tempered fractional Laplacian$(\Delta+\lambda)^{\frac{\beta}{2}}$ with singular integral in two dimension. We prove that for $u\in C^2(\mathbb{R}^2)$, the accuracy is $O(h^{2-\beta})$. Moreover, we solve the tempered fractional Poisson equation with Dirichlet boundary conditions and derive the error estimates for the present method. Numerical experiments on given exact solution verify the expected convergence rates.

% \PACS{PACS code1 \and PACS code2 \and more}
% \subclass{MSC code1 \and MSC code2 \and more}
\end{abstract}

%\newpage

\section{Introduction}

  Anomalous diffusion refers to the movements of particles whose trajectories' second moment is a nonlinear function of the time $t$ \cite{metzler}, being widely observed in the natural world \cite{Klafter2005} and having many applications in various fields, such as physical systems \cite{Hilfer2000}, stochastic dynamics \cite{Bogdan2003}, finance \cite{Mainardi}, image processing \cite{Buades} and so on. The fractional Laplacian $\Delta^{\beta/2}$ is the fundamental non-local operator for modelling anomalous dynamics, introduced as the infinitesimal generator of a $\beta$-stable L\'evy process  \cite{Applebaum2009,Gunzburger2013,Pozrikidis2016}, being the scaling limit of the L\'evy flight. The extremely long jumps make the second and all higher order moments of the L\'evy flight diverge, sometimes failing to well model some practically physical processes. To overcome this, a trivial idea is to introduce a parameter $\lambda$ (a sufficiently small number) to exponentially temper the isotropic power law measure of the jump length; the new processes generate the tempered fractional Laplacian $(\Delta+\lambda)^{\frac{\beta}{2}}$, being physically introduced and mathematically defined in \cite{Deng:17} with its definition
  %  The reference  \cite{Deng:17} introduce the tempered fractional Laplacian $(\Delta+\lambda)^{\frac{\beta}{2}}$ physically and define it mathematically. The first definition mathematically is given by the singular integral
\begin{equation}\label{idef1}
(\Delta+\lambda)^{\frac{\beta}{2}}u(\mathbf{x})=-c_{n,\beta,\lambda}{\rm P.V.}\int_{\mathbb{R}^n}\frac{u(\mathbf{x})-u(\mathbf{y})}{e^{\lambda|\mathbf{x}-\mathbf{y}|}|\mathbf{x}-\mathbf{y}|^{n+\beta}}d\mathbf{y} ~~~~~{\rm for}~~\beta \in(0,2),
\end{equation}
where
\begin{equation}
c_{n,\beta,\lambda}=
\frac{\Gamma(\frac{n}{2})}{2\pi^{n/2}|\Gamma(-\beta)|},
\end{equation}
and P.V. denotes the principal value integral, and $\Gamma(t)=\int_{0}^{\infty}s^{t-1}e^{-s}ds$ is the Gamma function; as to its Fourier transform \cite{Deng:17}, there is
 \begin{equation}\label{idef2}
 \begin{split}
 &\mathcal{F}\left((\Delta +\lambda)^{\beta/2}u(\mathbf{x})\right)\\
 &=(-1)^{\lfloor\beta\rfloor}\left(\lambda^\beta-(\lambda^2+|\mathbf{k}|^2)^\frac{\beta}{2}~_2F_1\left(-\frac{\beta}{2},\frac{n+\beta-1}{2};\frac{n}{2};\frac{|\mathbf{k}|^2}{\lambda^2+|\mathbf{k}|^2}\right)\right)\mathcal{F}({u}\left(\mathbf{x})\right),
 \end{split}
 \end{equation}
 where $\beta\in (0,1)\bigcup(1,2)$, $\lfloor\beta\rfloor$ means the biggest integer, being smaller than or equal to $\beta$, and $_2F_1$ is the Gauss hypergeometric function \cite{Abramowitz}.
Evidently, when $\lambda=0$, the expression (\ref{idef1}) reduces to the fractional Laplacian in the singular integral form \cite{Stein1970,Kwa2015}
\begin{equation}\label{idef3}
(\Delta)^{\frac{\beta}{2}}u(\mathbf{x})=-c_{n,\beta} {\rm P.V.}\int_{\mathbb{R}^n}\frac{u(\mathbf{x})-u(\mathbf{y})}{|\mathbf{x}-\mathbf{y}|^{n+\beta}}d\mathbf{y} ~~~~~{\rm for} ~~\beta \in(0,2),
\end{equation}
where
\begin{equation}
  c_{n,\beta}=\frac{\beta\Gamma(\frac{n+\beta}{2})}{2^{1-\beta}\pi^{n/2}\Gamma(1-\beta/2)}.
\end{equation}

The main challenge for numerically solving (\ref{idef1}) and (\ref{idef3}) comes from their non-locality and weak singularity, especially in high dimensional cases. Currently, fractional Laplacian is the trendy and hot topic in both mathematical and numerical fields.
%Because the non-locality and singularity of the expression (\ref{idef1}) and (\ref{idef3}), the numerical methods still remain challenging. So far, there have been some numerical methods for the fractional Laplacian in the singular integral form (\ref{idef3}).
For example, \cite{ACOSTA2017} introduces the finite element approximation for the $n$-dimensional Dirichlet homogeneous problem about fractional Laplacian and \cite{Acosta2017} presents the code employed for implementation in two dimension;  \cite{Huang2014} provides a finite difference-quadrature approach and gives its convergence proof; \cite{Huang1611} proposes several finite difference discretizations and tackles the non-locality, singularity and flat tails in practical implementations; \cite{Duo2017} provides a weighted trapezoidal rule for the fractional Laplacian in the singular integral form and gives the additional insights into the convergence behaviour of the method by the extensive numerical examples. For the tempered fractional Laplacian (\ref{idef1}), the existing numerical methods at present are mainly analyzed in one dimension. Among them, \cite{ZhangDeng2017} presents a Riesz basis Galerkin method for the tempered fractional Laplacian and gives the well-posedness proof of the Galerkin weak formulation and convergence analysis; \cite{Zhang2017} proposes a finite difference scheme and proves that the accuracy depends  on the regularity of the exact solution on $\bar{\Omega}$ rather than the regularity on the whole line. So far, its seems that there are no numerical analysis and implementation discussion on  (\ref{idef1}) in two dimension.

%To the best of our knowledge, analysis and implementation of the tempered fractional Laplacian (\ref{idef1}) in two dimension by the finite difference method is rare.

In this paper, we derive a finite difference scheme for the tempered fractional Laplacian (\ref{idef1}) in two dimension, based on the weighted trapezoidal rule combined with the bilinear interpolation. To be specific, we first write (\ref{idef1}) as the weighted integral of a weak singular function by introducing the function $\phi_{\gamma}$ and transforming the integration over the whole plane to the one in the first quadrant by symmetry; then we approximate the integration by the weighted trapezoidal rule in the neighborhood of any fixed point $(x,y)$ and by bilinear interpolation for the rest of the computational domain $\Omega$. It's worth mentioning that the present method also works well for the two dimensional fractional Laplacian (\ref{idef3}). Furthermore, we apply the discretization to solve the two dimensional tempered fractional Poisson equation with Dirichlet boundary conditions \cite{Deng:17}

\begin{equation}\label{defequ1}
\left\{
\begin{split}
  -(\Delta+\lambda)^{\frac{\beta}{2}}u(\mathbf{x})&=f(\mathbf{x}) & {\rm for}~\mathbf{x}\in\Omega \\
  u(\mathbf{x})&=0 & {\rm for}~\mathbf{x}\in\mathbb{R}^2\backslash \Omega.
 \end{split}
 \right.
 \end{equation}
The accuracy of the scheme is proved to be $O(h^{2-\beta})$ for $u\in C^2(\mathbb{R}^2)$.

%We prove that the accuracy of the method is $O(h^{2-\beta})$ for $u\in C^2(\mathbb{R}^2)$, according to the truncation error of the bilinear interpolation.

   As is well-known, it generally gives rise to a full matrix when discretizing the non-local operator. Therefore, the design of efficient iteration scheme makes more sense. When discretizing two dimensional tempered fractional Laplacian, we get a symmetric block Toeplitz matrix with Toeplitz block. Here, we use  the structure of the matrix to design the solver algorithm to (\ref{defequ1}). That is, we use Conjugate Gradient iterator to solve  (\ref{defequ1}); and in iteration process, we calculate the $B\mathbf{U}$ ($B$ is a symmetric block Toeplitz matrix with Toeplitz block and $\mathbf{U}$ is a vector) by fast Fourier transform \cite{Chen2005} to reduce the computational complexity. This algorithm has a memory requirement of $O(N^2)$ and a computational cost of $O(N^2 \log N^2)$ instead of a memory requirement of $O(N^4)$ and a computational cost of $O(N^6)$ per iteration. Next, to verify the convergence rates of the presented scheme, numerical experiments are performed for the equation with exact solution.
   %   on known exact solution validate the predicted convergence rates for $u\in C^2(\mathbb{R}^2)$,
%      and we also confirm the effectiveness for the fractional Laplacian of our method.
   For the unknown source term, we give an algorithm to approximate it, which changes the unbounded integration domain into bounded one through polar coordinate transformation in some special cases. For the details, see Appendix A. And we state the key points for the code implementation in Appendix B.

The paper is organized as follows. In Section 2, we propose a discretization scheme for the tempered fractional Laplacian through the weighted trapezoidal rule combined with the bilinear interpolation, and give its truncation error. In Section 3, we solve the tempered fractional Poisson equation with Dirichlet boundary conditions by the presented scheme and provide the error estimates. In the last Section, through numerical experiments for the equation with/without known solution, we verify the convergence rates and show the effectiveness of the schemes.

% we verify the convergence rates of our method by numerical examples and discuss the influence of parameters on the convergence orders.

\section{Numerical discretization of the tempered fractional Laplacian and its truncation error}
This section provides the discretization of the two dimensional tempered fractional Laplacian by the weighted trapezoidal rule combined with the bilinear interpolation on a bounded domain $\Omega=(-l,l)\times(-l,l)$ with extended homogeneous Dirichlet boundary conditions: $u(x,y)\equiv 0$ for $(x,y)\in\Omega^c$. Afterwards, we analyze the truncation error of the discretization.

Let us introduce the inner product and norms that will be used in the paper. Define the discrete $L_2$ inner product and $L_2$ norm as
\begin{equation}
\begin{split}
  &(\mathbf{V},\mathbf{W})=h\sum_{i=1}^{M}v_iw_i,\\
  &\|\mathbf{V}\|=\sqrt{(\mathbf{V},\mathbf{V})};
\end{split}
\end{equation}
denote
\begin{equation}
\begin{split}
 &\|v\|_{L_\infty(\Omega)}=\sup_{x\in \Omega}|v(x)|, \\
 &\|\mathbf{V}\|_\infty=\max_{1\leq i\leq M}|v_i|,
\end{split}
\end{equation}
as the continuous and discrete $L_\infty$ norm, where $\mathbf{V}, \mathbf{W} \in \mathbb{R}^M$.
\subsection{Numerical scheme}
According to (\ref{idef1}), the definition of the tempered fractional Laplacian in two dimension is \begin{equation}\label{def2}
\begin{split}
-(\Delta+\lambda)^{\frac{\beta}{2}}u(x,y)=-c_{2,\beta,\lambda}{\rm P.V.}\int\int_{\mathbb{R}^2}\frac{u(x+\xi,y+\eta)-u(x,y)}{e^{\lambda\sqrt{\xi^2+\eta^2}}\left(\sqrt{\xi^2+\eta^2}\right)^{{2+\beta}}}d\xi d\eta,
\end{split}
\end{equation}
%We can symmetrize (\ref{def2}) to obtain
which can be symmetrized as
\begin{equation}\label{nonintegral}
\begin{split}
&-(\Delta+\lambda)^{\frac{\beta}{2}}u(x,y)\\
=&-\frac{c_{2,\beta,\lambda}}{2}\int\int_{\mathbb{R}^2}\frac{u(x+\xi,y+\eta)-2u(x,y)+u(x-\xi,y-\eta)}{e^{\lambda\sqrt{\xi^2+\eta^2}}\left(\sqrt{\xi^2+\eta^2}\right)^{{2+\beta}}}d\xi d\eta\\
=&-\frac{c_{2,\beta,\lambda}}{4}\int\int_{\mathbb{R}^2}\frac{g(x,y,\xi,\eta)}{e^{\lambda\sqrt{\xi^2+\eta^2}}\left(\sqrt{\xi^2+\eta^2}\right)^{{2+\beta}}}d\xi d\eta
\end{split}
\end{equation}
with
\begin{equation}
g(x,y,\xi,\eta)=u(x+\xi,y+\eta)+u(x-\xi,y+\eta)+u(x-\xi,y-\eta)+u(x+\xi,y-\eta)-4u(x,y).
\end{equation}
 By the symmetry of the integral domain and integrand, Eq. (\ref{nonintegral}) can be rewritten as
\begin{equation}
\begin{split}\label{decequa1}
&-(\Delta+\lambda)^{\frac{\beta}{2}}u(x,y)
=-c_{2,\beta,\lambda}\int_{0}^{\infty}\int_{0}^{\infty}\frac{g(x,y,\xi,\eta)}{e^{\lambda\sqrt{\xi^2+\eta^2}}\left(\sqrt{\xi^2+\eta^2}\right)^{{2+\beta}}}d\eta d\xi .\\
\end{split}
\end{equation}
If we denote
\begin{equation}
\phi_{\gamma}(\xi,\eta)=\frac{g(x,y,\xi,\eta)}{e^{\lambda\sqrt{\xi^2+\eta^2}}\left(\sqrt{\xi^2+\eta^2}\right)^{{\gamma}}},
\end{equation}
where $\gamma\in(\beta,2]$, then (\ref{decequa1}) becomes
\begin{equation}
\begin{split}\label{decequa2}
&-(\Delta+\lambda)^{\frac{\beta}{2}}u(x,y)
=-c_{2,\beta,\lambda}\int_{0}^{\infty}\int_{0}^{\infty}\frac{\phi_{\gamma}(\xi,\eta)}{\left(\sqrt{\xi^2+\eta^2}\right)^{-\gamma+2+\beta}} d\eta d\xi.
\end{split}
\end{equation}
Now, we just need to discretize the tempered fractional Laplacian in $[0,\infty)\times[0,\infty)$ instead of $\mathbb{R}\times\mathbb{R}$. Taking a constant $L=2l$, we have $u(x+\xi,y+\eta)=0$ for $(\xi,\eta)\notin (-L,L)\times(-L,L)$. Thus,
\begin{equation}\label{equdis28}
\begin{split}
-(\Delta+\lambda)^{\frac{\beta}{2}}u(x,y)=-c_{2,\beta,\lambda}&\Big(\int_{0}^{L}\int_{0}^{L}\frac{\phi_{\gamma}(\xi,\eta)}{\left(\sqrt{\xi^2+\eta^2}\right)^{-\gamma+2+\beta}} d\eta d\xi\\
&-4\int_{0}^{L}\int_{L}^{\infty}\frac{u(x,y)}{e^{\lambda\sqrt{\xi^2+\eta^2}}\left(\sqrt{\xi^2+\eta^2}\right)^{{2+\beta}}} d\eta d\xi\\
&-4\int_{L}^{\infty}\int_{0}^{L}\frac{u(x,y)}{e^{\lambda\sqrt{\xi^2+\eta^2}}\left(\sqrt{\xi^2+\eta^2}\right)^{{2+\beta}}} d\eta d\xi\\
&-4\int_{L}^{\infty}\int_{L}^{\infty}\frac{u(x,y)}{e^{\lambda\sqrt{\xi^2+\eta^2}}\left(\sqrt{\xi^2+\eta^2}\right)^{{2+\beta}}}d\eta d\xi \Big).
\end{split}
\end{equation}
For convenience, we denote
\begin{equation}
  \begin{split}
    G^{\infty}=&\int_{0}^{L}\int_{L}^{\infty}\frac{1}{e^{\lambda\sqrt{\xi^2+\eta^2}}\left(\sqrt{\xi^2+\eta^2}\right)^{{2+\beta}}}d\eta d\xi\\
&+\int_{L}^{\infty}\int_{0}^{L}\frac{1}{e^{\lambda\sqrt{\xi^2+\eta^2}}\left(\sqrt{\xi^2+\eta^2}\right)^{{2+\beta}}} d\eta d\xi\\
&+\int_{L}^{\infty}\int_{L}^{\infty}\frac{1}{e^{\lambda\sqrt{\xi^2+\eta^2}}\left(\sqrt{\xi^2+\eta^2}\right)^{{2+\beta}}}d\eta d\xi.\\
  \end{split}
\end{equation}

Let the mesh size $h_1=L/N_{i},h_2=L/N_{j}$; denote grid points $\xi_i=ih_1, \eta_j=jh_2$, for $1\leq i \leq N_{i}, 1\leq j \leq N_{j}$; for convenience, we set $N_i=N_j$. Then, we can formulate the first integral in (\ref{equdis28}) as
\begin{equation}\label{equtodis}
\int_{0}^{L}\int_{0}^{L}\phi_{\gamma}(\xi,\eta)(\xi^2+\eta^2)^{\frac{\gamma-2-\beta}{2}} d\eta d\xi =\sum_{i=0}^{N_i-1}\sum_{j=0}^{N_j-1}\int_{\xi_{i}}^{\xi_{i+1}}\int_{\eta_{j}}^{\eta_{j+1}}\phi_{\gamma}(\xi,\eta)(\xi^2+\eta^2)^{\frac{\gamma-2-\beta}{2}}d\eta d\xi.
\end{equation}

 For (\ref{equtodis}), when $(i,j)=(0,0)$, it is easy to see that the integration is weak singular. So we approximate the integral by the weighted trapezoidal rule. For different $\gamma$, we use different integral nodes to approximate it, namely,
\begin{equation}\label{equdis11}
\begin{split}
\int_{\xi_{0}}^{\xi_1}\int_{\eta_0}^{\eta_1}&\phi_{\gamma}(\xi,\eta)(\xi^2+\eta^2)^{\frac{\gamma-2-\beta}{2}} d\eta d\xi=\\&\left\{
    \begin{split}
      &\frac{1}{4}\left(\lim_{(\xi,\eta)\rightarrow(0,0)}\phi_{\gamma}(\xi,\eta)+\phi_{\gamma}(\xi_0,\eta_1)+\phi_{\gamma}(\xi_1,\eta_1)+\phi_{\gamma}(\xi_1,\eta_0)\right)G_{0,0},~~\gamma\in(\beta,2);\\
      &\frac{1}{3}\left(\phi_{\gamma}(\xi_0,\eta_1)+\phi_{\gamma}(\xi_1,\eta_1)+\phi_{\gamma}(\xi_1,\eta_0)\right)G_{0,0},~~~~\gamma=2,
    \end{split}
    \right.
    \end{split}
\end{equation}
where
\begin{equation}\label{equdefG00}
G_{0,0}=\int_{\xi_{0}}^{\xi_{1}}\int_{\eta_{0}}^{\eta_{1}}(\xi^2+\eta^2)^{\frac{\gamma-2-\beta}{2}} d\eta d\xi.
\end{equation}
Assuming $u$ is smooth enough, for $\gamma \in(\beta,2)$, there exists
\begin{equation}
  \lim_{(\xi,\eta)\rightarrow(0,0)}\phi_{\gamma}(\xi,\eta)=0,
\end{equation}
so we introduce a parameter $k_\gamma$
\begin{equation}
  k_\gamma=\left\{
  \begin{split}
  1~~~~~~~~~~~~~~&\gamma\in(\beta,2);\\
  \frac{4}{3}~~~~~~~~~~~~~~&\gamma=2.
  \end{split}
  \right.
\end{equation}
Then, Eq. (\ref{equdis11}) can be rewritten as
\begin{equation}
  \int_{\xi_{0}}^{\xi_1}\int_{\eta_0}^{\eta_1}\phi_{\gamma}(\xi,\eta)(\xi^2+\eta^2)^{\frac{\gamma-2-\beta}{2}}d\eta d\xi=\frac{k_\gamma}{4}\left(\phi_{\gamma}(\xi_0,\eta_1)+\phi_{\gamma}(\xi_1,\eta_1)+\phi_{\gamma}(\xi_1,\eta_0)\right)G_{0,0}.
\end{equation}
For another part of (\ref{equtodis}), when $(i,j)\neq(0,0)$, we deal with the integration by the bilinear interpolation. Before discretizing it, we define the following functions
\begin{equation}\label{equdefG}
\begin{split}
G_{i,j}&=\frac{1}{h^2}\int_{\xi_{i}}^{\xi_{i+1}}\int_{\eta_{j}}^{\eta_{j+1}}(\xi^2+\eta^2)^{\frac{\gamma-2-\beta}{2}} d\eta d\xi,\\
G^\xi_{i,j}&=\frac{1}{h^2}\int_{\xi_{i}}^{\xi_{i+1}}\int_{\eta_{j}}^{\eta_{j+1}}\xi(\xi^2+\eta^2)^{\frac{\gamma-2-\beta}{2}} d\eta d\xi,\\
G^\eta_{i,j}&=\frac{1}{h^2}\int_{\xi_{i}}^{\xi_{i+1}}\int_{\eta_{j}}^{\eta_{j+1}}\eta(\xi^2+\eta^2)^{\frac{\gamma-2-\beta}{2}}d\eta d\xi, \\
G^{\xi\eta}_{i,j}&=\frac{1}{h^2}\int_{\xi_{i}}^{\xi_{i+1}}\int_{\eta_{j}}^{\eta_{j+1}}\xi\eta(\xi^2+\eta^2)^{\frac{\gamma-2-\beta}{2}} d\eta d\xi,
\end{split}
\end{equation}
where $G_{i,j}$, $G^\xi_{i,j}$, $G^\eta_{i,j}$, $G^{\xi\eta}_{i,j}$ can be obtained by numerical integration.

Further denote $I_{i,j}$ as the interpolation integration in $[\xi_i,\xi_{i+1}]\times [\eta_j,\eta_{j+1}]$, i.e.,
\begin{equation}
  \begin{split}
      I_{i,j}=&\phi_\gamma(\xi_i,\eta_j)(G^{\xi\eta}_{i,j}-\xi_{i+1}G^\eta_{i,j}-\eta_{j+1}G^\xi_{i,j}+\xi_{i+1}\eta_{j+1}G_{i,j})\\
                        &-\phi_\gamma(\xi_{i+1},\eta_j)(G^{\xi\eta}_{i,j}-\xi_{i}G^\eta_{i,j}-\eta_{j+1}G^\xi_{i,j}+\xi_{i}\eta_{j+1}G_{i,j})\\
                        &-\phi_\gamma(\xi_i,\eta_{j+1})(G^{\xi\eta}_{i,j}-\xi_{i+1}G^\eta_{i,j}-\eta_{j}G^\xi_{i,j}+\xi_{i+1}\eta_{j}G_{i,j})\\
                        &+\phi_\gamma(\xi_{i+1},\eta_{j+1})(G^{\xi\eta}_{i,j}-\xi_{i}G^\eta_{i,j}-\eta_{j}G^\xi_{i,j}+\xi_{i}\eta_{j}G_{i,j});
  \end{split}
\end{equation}
and let
\begin{equation}\label{equdefW}
  \begin{split}
  W^1_{i,j}&=G^{\xi\eta}_{i,j}-\xi_{i+1}G^\eta_{i,j}-\eta_{j+1}G^\xi_{i,j}+\xi_{i+1}\eta_{j+1}G_{i,j},\\
  W^2_{i,j}&=-\left(G^{\xi\eta}_{i-1,j}-\xi_{i-1}G^\eta_{i-1,j}-\eta_{j+1}G^\xi_{i-1,j}+\xi_{i-1}\eta_{j+1}G_{i-1,j}\right),\\
  W^3_{i,j}&=-\left(G^{\xi\eta}_{i,j-1}-\xi_{i+1}G^\eta_{i,j-1}-\eta_{j-1}G^\xi_{i,j-1}+\xi_{i+1}\eta_{j-1}G_{i,j-1}\right),\\
  W^4_{i,j}&=G^{\xi\eta}_{i-1,j-1}-\xi_{i-1}G^\eta_{i-1,j-1}-\eta_{j-1}G^\xi_{i-1,j-1}+\xi_{i-1}\eta_{j-1}G_{i-1,j-1}.
  \end{split}
\end{equation}
Then, $I_{i,j}$ can be rewritten as
\begin{equation}\label{equIij2}
  \begin{split}
      I_{i,j}=&\phi_\gamma(\xi_i,\eta_j)W^1_{i,j}+\phi_\gamma(\xi_{i+1},\eta_j)W^2_{i+1,j}+\phi_\gamma(\xi_i,\eta_{j+1})W^3_{i,j+1}
                        +\phi_\gamma(\xi_{i+1},\eta_{j+1})W^4_{i+1,j+1},
  \end{split}
\end{equation}
and Eq. (\ref{equtodis}) becomes
\begin{equation}\label{equdiswithI}
\begin{split}
&\sum_{i=0}^{N_i-1}\sum_{j=0}^{N_j-1}\int_{\xi_{i}}^{\xi_{i+1}}\int_{\eta_{j}}^{\eta_{j+1}}\phi_{\gamma}(\xi,\eta)(\xi^2+\eta^2)^{\frac{\gamma-2-\beta}{2}}d\eta d\xi \\
=&\frac{k_\gamma}{4}\left(\phi_{\gamma}(\xi_0,\eta_1)+\phi_{\gamma}(\xi_1,\eta_1)+\phi_{\gamma}(\xi_1,\eta_0)\right)G_{0,0}\\
&+\sum^{i=N_i-1,j=N_j-1}_{
\begin{subarray}{c}
i,j=0;\\(i,j)\neq(0,0)
\end{subarray}}I_{i,j}.
\end{split}
\end{equation}
Combining (\ref{equIij2}) with (\ref{equdiswithI}), we derive
\begin{equation}
  \begin{split}
    &\sum_{i=0}^{N_i-1}\sum_{j=0}^{N_j-1}\int_{\xi_{i}}^{\xi_{i+1}}\int_{\eta_j}^{\eta_{j+1}}\phi_{\gamma}(\xi,\eta)(\xi^2+\eta^2)^{\frac{\gamma-2-\beta}{2}} d\eta d\xi \\
    =&\left(\frac{k_\gamma}{4}G_{0,0}+W^1_{1,1}+W^2_{1,1}+W^3_{1,1}\right)\phi_{\gamma}(\xi_1,\eta_1)\\
    &+\left(\frac{k_\gamma}{4}G_{0,0}+W^1_{1,0}\right)\phi_{\gamma}(\xi_1,\eta_0)+\left(\frac{k_\gamma}{4}G_{0,0}+W^1_{0,1}\right)\phi_{\gamma}(\xi_0,\eta_1)\\
    &+\sum_{i=2}^{N_i-1}\left(W^1_{i,0}+W^2_{i,0}\right)\phi_{\gamma}(\xi_i,\eta_0)+\sum_{j=2}^{N_j-1}\left(W^1_{0,j}+W^3_{0,j}\right)\phi_{\gamma}(\xi_0,\eta_j)\\
    &+\sum_{i=1}^{N_i-1}\left(W^3_{i,N_j}+W^4_{i,N_j}\right)\phi_{\gamma}(\xi_i,\eta_{N_j})+\sum_{j=1}^{N_j-1}\left(W^2_{N_i,j}+W^4_{N_i,j}\right)\phi_{\gamma}(\xi_{N_i},\eta_j)\\
    &+W^3_{0,N_j}\phi_{\gamma}(\xi_0,\eta_{N_j})+W^2_{N_i,0}\phi_{\gamma}(\xi_{N_i},\eta_{0})+W^4_{N_i,N_j}\phi_{\gamma}(\xi_{N_i},\eta_{N_j})\\
    &+\sum^{i=N_i-1,j=N_j-1}_{
        \begin{subarray}{c}
            i,j=1;\\(i,j)\neq(1,1)
        \end{subarray}}\left(W^1_{i,j}+W^2_{i,j}+W^3_{i,j}+W^4_{i,j}\right)\phi_{\gamma}(\xi_i,\eta_j).
  \end{split}
\end{equation}
For the second part of (\ref{equdis28}), namely $G^\infty$, we get it by numerical integration.

Denote $u_{p,q}=u(-l+ph,-l+qh)$, $(p,q\in \mathbb{Z})$. Then we can get the discretization scheme
\begin{equation}\label{defdis}
  -(\Delta+\lambda)_h^{\beta/2}u_{p,q}=\sum_{i=-N_i}^{i=N_i}\sum_{j=-N_j}^{j=N_j}w_{|i|,|j|}u_{p-i,q-j},
\end{equation}
where
\begin{equation}\label{equweightoffl}\footnotesize
w_{i,j}=-c_{2,\beta,\lambda}\left\{
\begin{split}
&-4\left(\frac{\frac{k_\gamma}{4}G_{0,0}+W^1_{1,1}+W^2_{1,1}+W^3_{1,1}}{e^{\lambda\sqrt{\xi_1^2+\eta_1^2}}\left(\sqrt{\xi_1^2+\eta_1^2}\right)^\gamma}\right.&\\
&~~~~~~   +\frac{\frac{k_\gamma}{4}G_{0,0}+W^1_{1,0}}{e^{\lambda\sqrt{\xi_1^2+\eta_0^2}}\left(\sqrt{\xi_1^2+\eta_0^2}\right)^\gamma}+\frac{\frac{k_\gamma}{4}G_{0,0}+W^1_{0,1}}{e^{\lambda\sqrt{\xi_0^2+\eta_1^2}}\left(\sqrt{\xi_0^2+\eta_1^2}\right)^\gamma}&\\
&~~~~~~   +\sum_{i=2}^{N_i-1}\frac{W^1_{i,0}+W^2_{i,0}}{e^{\lambda\sqrt{\xi_i^2+\eta_0^2}}\left(\sqrt{\xi_i^2+\eta_0^2}\right)^\gamma}+\sum_{j=2}^{N_j-1}\frac{W^1_{0,j}+W^3_{0,j}}{e^{\lambda\sqrt{\xi_0^2+\eta_j^2}}\left(\sqrt{\xi_0^2+\eta_j^2}\right)^\gamma}&\\
&~~~~~~   +\sum_{i=2}^{N_i-1}\frac{W^3_{i,N_j}+W^4_{i,N_j}}{e^{\lambda\sqrt{\xi_i^2+\eta_{N_j}^2}}\left(\sqrt{\xi_i^2+\eta_{N_j}^2}\right)^\gamma}+\sum_{j=2}^{N_j-1}\frac{W^2_{N_i,j}+W^4_{N_i,j}}{e^{\lambda\sqrt{\xi_{N_i}^2+\eta_j^2}}\left(\sqrt{\xi_{N_i}^2+\eta_j^2}\right)^\gamma}&\\
&~~~~~~+\sum_{i=1,j=1,(i,j)\neq(1,1)}^{i=N_i-1,j=N_j-1}\frac{W^1_{i,j}+W^2_{i,j}+W^3_{i,j}+W^4_{i,j}}{e^{\lambda\sqrt{\xi_{i}^2+\eta_j^2}}\left(\sqrt{\xi_{i}^2+\eta_j^2}\right)^\gamma}&\\
&~~~~~~   +\frac{W^3_{0,N_j}}{e^{\lambda\sqrt{\xi_0^2+\eta_{N_j}^2}}\left(\sqrt{\xi_0^2
+\eta_{N_j}^2}\right)^\gamma}+\frac{W^2_{N_i,0}}{e^{\lambda\sqrt{\xi_{N_i}^2+\eta_0^2}}\left(\sqrt{\xi_{N_i}^2+\eta_0^2}\right)^\gamma}&\\
&~~~~~~\left.+\frac{W^4_{N_i,N_j}}{e^{\lambda\sqrt{\xi_{N_i}^2+\eta_{N_j}^2}}\left(\sqrt{\xi_{N_i}^2+\eta_{N_j}^2}\right)^\gamma}+G^\infty\right),&i=0,j=0\\
&\frac{\frac{k_\gamma}{4}G_{0,0}+W^1_{1,1}+W^2_{1,1}+W^3_{1,1}}{e^{\lambda\sqrt{\xi_1^2+\eta_1^2}}\left(\sqrt{\xi_1^2+\eta_1^2}\right)^\gamma},&i=1,j=1\\
&2\frac{\frac{k_\gamma}{4}G_{0,0}+W^1_{1,0}}{e^{\lambda\sqrt{\xi_1^2+\eta_0^2}}\left(\sqrt{\xi_1^2+\eta_0^2}\right)^\gamma},&i=1,j=0\\
&2\frac{\frac{k_\gamma}{4}G_{0,0}+W^1_{0,1}}{e^{\lambda\sqrt{\xi_0^2+\eta_1^2}}\left(\sqrt{\xi_0^2+\eta_1^2}\right)^\gamma},&i=0,j=1\\
&2\frac{W^1_{i,0}+W^2_{i,0}}{e^{\lambda\sqrt{\xi_i^2+\eta_0^2}}\left(\sqrt{\xi_i^2+\eta_0^2}\right)^\gamma,}&1<i<N_i,j=0\\
&2\frac{W^1_{0,j}+W^3_{0,j}}{e^{\lambda\sqrt{\xi_0^2+\eta_j^2}}\left(\sqrt{\xi_0^2+\eta_j^2}\right)^\gamma},&i=0, 1<j<N_j\\
&\frac{W^3_{i,N_j}+W^4_{i,N_j}}{e^{\lambda\sqrt{\xi_i^2+\eta_{N_j}^2}}\left(\sqrt{\xi_i^2+\eta_{N_j}^2}\right)^\gamma},&1<i<N_i,j=N_j\\
&\frac{W^2_{N_i,j}+W^4_{N_i,j}}{e^{\lambda\sqrt{\xi_{N_i}^2+\eta_j^2}}\left(\sqrt{\xi_{N_i}^2+\eta_j^2}\right)^\gamma},&i=N_i, 1<j<N_j\\
&2\frac{W^3_{0,N_j}}{e^{\lambda\sqrt{\xi_{0}^2+\eta_{N_j}^2}}\left(\sqrt{\xi_{0}^2+\eta_{N_j}^2}\right)^\gamma},& i=0,j=N_j\\
&2\frac{W^2_{N_i,0}}{e^{\lambda\sqrt{\xi_{N_i}^2+\eta_0^2}}\left(\sqrt{\xi_{N_i}^2+\eta_0^2}\right)^\gamma},& i=N_i,j=0\\
&\frac{W^4_{N_i,N_j}}{e^{\lambda\sqrt{\xi_{N_i}^2+\eta_{N_j}^2}}\left(\sqrt{\xi_{N_i}^2+\eta_{N_j}^2}\right)^\gamma},&i=N_i,j=N_j\\
&\frac{W^1_{i,j}+W^2_{i,j}+W^3_{i,j}+W^4_{i,j}}{e^{\lambda\sqrt{\xi_{i}^2+\eta_j^2}}\left(\sqrt{\xi_{i}^2+\eta_j^2}\right)^\gamma},&otherwise
\end{split}
\right.
\end{equation}

For the sake of convenience, we write the matrix form of the scheme (\ref{defdis}) as
\begin{equation}\label{desdif2}
-(\Delta+\lambda)_{h}^{\frac{\beta}{2}}\mathbf{U}=B\mathbf{U},
\end{equation}
 where
 \begin{equation}
  \mathbf{U}=\left(u_{1,1},u_{1,2},\cdots,u_{1,N_j-1},u_{2,1}\cdots,u_{2,N_j-1},\cdots,u_{N_i-1,N_j-1}\right)^{T},
  \end{equation}
  and
   %$B$ is the matrix representation of the tempered fractional Laplacian,
 \begin{equation}
 B=\left[\begin{matrix}
 w_{|1-1|,|1-1|}& w_{|1-1|,|2-1|} &\cdots & w_{|(N_i-1)-1|,|(N_j-1)-1|} \\
 w_{|1-1|,|1-2|}& w_{|1-1|,|2-2|} &\cdots & w_{|(N_i-1)-1|,|(N_j-1)-2|} \\
 \vdots& \vdots &\ddots & \vdots \\
 w_{|1-(N_i-1)|,|1-(N_j-1)|}& w_{|1-(N_i-1)|,|2-(N_j-1)|} &\cdots & w_{|(N_i-1)-(N_i-1)|,|(N_j-1)-(N_j-1)|} \\
 \end{matrix}\right],
 \end{equation}
 is the matrix representation of the tempered fractional Laplacian.

 Denote the numerical solution of Eq. (\ref{defequ1}) at $(-l+ph,-l+qh)$ as $u^h_{p,q}$ and the source term $F$ at $(-l+ph,-l+qh)$ as $f_{p,q}$, $(p,q\in \mathbb{Z})$. Then Eq. (\ref{defequ1}) can also be written as
 \begin{equation}\label{matex1}
 B\mathbf{U}_h=F,
 \end{equation}
 where
  \begin{equation}
  \mathbf{U}_h=\left(u^h_{1,1},u^h_{1,2},\cdots,u^h_{1,N_j-1},u^h_{2,1}\cdots,u^h_{2,N_j-1},\cdots,u^h_{N_i-1,N_j-1}\right)^{T},
  \end{equation}
 and
 \begin{equation}
   F=\left(f_{1,1},f_{1,2},\cdots,f_{1,N_j-1},f_{2,1}\cdots,f_{2,N_j-1},\cdots,f_{N_i-1,N_j-1}\right)^{T}.
 \end{equation}
\subsection{Structure of the stiffness matrix $B$}
\begin{definition}\cite{Chen2005}
 The symmetric $N\times N$ matrix $T$ is called the symmetric  Toeplitz matrix if its entries are constant along each diagonal, i.e.,
 \begin{equation}
   T=\left[
   \begin{matrix}
     t_0&t_1&\cdots&t_{N-2}&t_{N-1}\\
     t_1&t_0&\cdots&~t_{N-3}&t_{N-2}\\
     \vdots& \vdots &\ddots & \vdots& \vdots \\
     t_{N-1}&t_{N-2}&\cdots&~t_1&t_0\\
   \end{matrix}\right].
 \end{equation}
 And the symmetric $N^2\times N^2$ matrix $H$ is called the symmetric block Toeplitz matrix with Toeplitz block, which has following structure
  \begin{equation}
   H=\left[
   \begin{matrix}
     T_0&T_1&\cdots&T_{n-2}&T_{n-1}\\
     T_1&T_0&\cdots&~T_{n-3}&T_{n-2}\\
     \vdots& \vdots &\ddots & \vdots& \vdots \\
     T_{n-1}&T_{n-2}&\cdots&~T_1&T_0\\
   \end{matrix}\right],
 \end{equation}
 where each $T_i$ is a symmetric Toeplitz matrix.
  \end{definition}

  Since a symmetric Toeplitz matrix $T$ is determined by its first column and each block of $H$ is symmetric Toeplitz matrix, we can store $H$ by a $N\times N$ matrix to reduce the memory requirement \cite{Chen2005}.
In our scheme (\ref{desdif2}), it is easy to verify that the matrix $B$ is a symmetric block Toeplitz matrix with Toeplitz block according to (\ref{equweightoffl}), so we store $B$ by a $N\times N$ matrix to reduce the memory requirement to $O(N^2)$. When solving $B\mathbf{U}_h=F$, the fast Fourier transform can be used in the iteration process and the computational cost of calculating $B\mathbf{U}$ ($\mathbf{U}\in R^{N^2}$ is a vector) can be reduced to $O(N^2 \log N^2)$.

\subsection{Truncation error}
\begin{lemma}\label{lemfunc2error}
Let $\beta\in(0,2)$ , $\xi> 0$ and $\eta>0$. If $u(x,y)\in C^{2}(\mathbb{R}^2)$, the derivative $D^{\alpha}\phi_{\gamma}$ ($\alpha$ is multi-index and $|\alpha|\leq2$) exists for any $\gamma\in (\beta,2]$, then for $(x,y)\in\Omega$, there are
\begin{equation}\label{eqphikz}
\begin{split}
&\left|\phi_{\gamma}\right|\leq C\left(\xi^2+\eta^2\right)^{1-\frac{\gamma}{2}},\\
&\left|\frac{\partial^2\phi_{\gamma}}{\partial \xi^2}\right|\leq C\left((\xi^2+\eta^2)^{-\frac{\gamma}{2}}+(\xi^2+\eta^2)^{\frac{1}{2}-\frac{\gamma}{2}}+(\xi^2+\eta^2)^{1-\frac{\gamma}{2}}\right),\\
&\left|\frac{\partial^2\phi_{\gamma}}{\partial \eta^2}\right|\leq C\left((\xi^2+\eta^2)^{-\frac{\gamma}{2}}+(\xi^2+\eta^2)^{\frac{1}{2}-\frac{\gamma}{2}}+(\xi^2+\eta^2)^{1-\frac{\gamma}{2}}\right)
\end{split}
\end{equation}
with $C$ being a positive constants.
\end{lemma}
\begin{proof}
  Using Taylor's formula, we obtain
  \begin{equation}
    \begin{split}
      &\left|\phi_{\gamma}(\xi,\eta)\right|\leq\left|\frac{g(x,y,\xi,\eta)}{(\xi^2+\eta^2)^\frac{\gamma}{2}}\right|\\
      \leq &\left|\frac{(\xi\frac{\partial}{\partial x}+\eta\frac{\partial}{\partial y})^2u\left|_{(x^*_1,y^*_1)}\right.+(-\xi\frac{\partial}{\partial x}+\eta\frac{\partial}{\partial y})^2u\left|_{(x^*_2,y^*_2)}\right.}{2!(\xi^2+\eta^2)^\frac{\gamma}{2}}\right.\\
      &+\left.\frac{(\xi\frac{\partial}{\partial x}-\eta\frac{\partial}{\partial y})^2u\left|_{(x^*_3,y^*_3)}\right.+(-\xi\frac{\partial}{\partial x}-\eta\frac{\partial}{\partial y})^2u\left|_{(x^*_4,y^*_4)}\right.}{2!(\xi^2+\eta^2)^\frac{\gamma}{2}}\right|\\
      \leq& C\left|\frac{\xi^2+\eta^2}{(\xi^2+\eta^2)^\frac{\gamma}{2}}\right|\\
      \leq&C(\xi^2+\eta^2)^{1-\frac{\gamma}{2}},
    \end{split}
  \end{equation}
  where
  \begin{equation}
    \begin{split}
      (x^*_1,y^*_1)&\in[x,x+\xi]\times[y,y+\eta],\\
      (x^*_2,y^*_2)&\in[x-\xi,x]\times[y,y+\eta],\\
      (x^*_3,y^*_3)&\in[x,x+\xi]\times[y-\eta,y],\\
      (x^*_4,y^*_4)&\in[x-\xi,x]\times[y-\eta,y].\\
    \end{split}
  \end{equation}
  For $\left|\frac{\partial^2\phi_{\gamma}}{\partial \xi^2}\right|$, we have
  \begin{equation}
    \begin{split}
      \left|\frac{\partial^2\phi_{\gamma}}{\partial \xi^2}\right|&\leq\left|\frac{ g^{(2,0)}(x,y,\xi,\eta)}{(\xi^2+\eta^2)^\frac{\gamma}{2}}\right|
      +C\left|\frac{g^{(1,0)}(x, y,\xi,\eta)}{(\xi^2+\eta^2)^{1+\frac{\gamma}{2}}}\xi\right|\\
      &+C\left|\frac{g^{(1,0)}(x, y,\xi,\eta)}{(\xi^2+\eta^2)^{\frac{1}{2}+\frac{\gamma}{2}}}\xi\right|
      +C\left|\frac{g(x,y,\xi,\eta)}{(\xi^2+\eta^2)^{1+\frac{\gamma}{2}}}\right|\\
      &+C\left|\frac{g(x,y,\xi,\eta)}{(\xi^2+\eta^2)^{\frac{1}{2}+\frac{\gamma}{2}}}\right|
      +C\left|\frac{g(x,y,\xi,\eta)}{(\xi^2+\eta^2)^{2+\frac{\gamma}{2}}}\xi^2\right|\\
      &+C\left|\frac{g(x,y,\xi,\eta)}{(\xi^2+\eta^2)^{\frac{3}{2}+\frac{\gamma}{2}}}\xi^2\right|
      +C\left|\frac{g(x,y,\xi,\eta)}{(\xi^2+\eta^2)^{1+\frac{\gamma}{2}}}\xi^2\right|.
    \end{split}
  \end{equation}
  Using Taylor's formula again leads to
  \begin{equation}\label{phi}
    \begin{split}
     \left|\frac{\partial^2\phi_{\gamma}}{\partial \xi^2}\right|\leq C\left((\xi^2+\eta^2)^{-\frac{\gamma}{2}}+(\xi^2+\eta^2)^{\frac{1}{2}-\frac{\gamma}{2}}+(\xi^2+\eta^2)^{1-\frac{\gamma}{2}}\right).
    \end{split}
  \end{equation}
  The estimate for $\left|\frac{\partial^2\phi_{\gamma}}{\partial \eta^2}\right|$ can be similarly obtained as the one for $\left|\frac{\partial^2\phi_{\gamma}}{\partial \xi^2}\right|$. Then the desired inequalities (\ref{eqphikz}) hold.
\end{proof}

Next, we introduce a lemma about the error of the bilinear interpolation.
\begin{lemma}\label{lemmaintepola}\cite{Mobner2009}
Let $I_h$ denote the bilinear interpolant on the box $K=[0,h]\times[0,h]$. For $f\in W^{2,\infty}(K)$ ($W^{k,p}(K)$ denotes a sobolev space), the error of bilinear interpolant is bounded by
\begin{equation}
  \|f-If\|_{L_\infty}\leq ch^2\left(\left\|\frac{\partial^2 f}{\partial x^2}\right\|_{L_\infty}+\left\|\frac{\partial^2 f}{\partial y^2}\right\|_{L_\infty}\right).
\end{equation}
\end{lemma}
\begin{proof}
The proof can be completed by using the tensor-product polynomial approximation given in \cite{Brenner2008}. We omit the details here.
%can be got from \cite{Brenner2008} about the tensor-product polynomial approximation, we omit the proof for simplicity here.
\end{proof}

\begin{theorem}\label{thmtrunct}
Denote $(\Delta+\lambda)^{\frac{\beta}{2}}_{h}$ as a finite difference approximation of the tempered fractional Laplacian $(\Delta+\lambda)^{\frac{\beta}{2}}$.
Suppose that $u(x,y)\in C^{2}(\mathbb{R}^2)$ has finite support on an open set $\Omega\subset\mathbb{R}^2$. Then, for any $\gamma\in(\beta,2]$, there is
\begin{equation}
\left\|(\Delta+\lambda)^{\frac{\beta}{2}}u(x,y)-(\Delta+\lambda)_{h}^{\frac{\beta}{2}}u(x,y)\right\|_{L_\infty(\Omega)}\leq Ch^{2-\beta},~~~~~{\rm for}~\beta\in (0,2)
\end{equation}
with $C$ being a positive constant depending on $\beta$ and $\gamma$.
\end{theorem}
\begin{proof}
From (\ref{decequa1}), (\ref{equdis28}), (\ref{equtodis}) and (\ref{equdiswithI}), we obtain the error function
\begin{equation}\label{error1}
\begin{split}
  e^h_{\beta,\gamma}(x,y)=&(\Delta+\lambda)^{\frac{\beta}{2}}u(x,y)-(\Delta+\lambda)_{h}^{\frac{\beta}{2}}u(x,y)\\
        =&\left(\int_{\xi_{0}}^{\xi_1}\int_{\eta_{0}}^{\eta_1}\phi_{\gamma}(\xi,\eta)(\xi^2+\eta^2)^{\frac{\gamma-2-\beta}{2}}d\eta d\xi\right.\\
        &\left.-\int_{\xi_{0}}^{\xi_1}\int_{\eta_{0}}^{\eta_1}\frac{k_\gamma}{4}\left(\phi_{\gamma}(\xi_0,\eta_1)+\phi_{\gamma}(\xi_1,\eta_0)+\phi_{\gamma}(\xi_1,\eta_1)\right)(\xi^2+\eta^2)^{\frac{\gamma-2-\beta}{2}}d\eta d\xi\right) \\
        &+\sum_{
        \begin{subarray}
         ~i=0;j=0;\\(i,j)\neq (0,0)
        \end{subarray}
        }^{i=N_i-1;j=N_j-1}\left(\int_{\xi_{i}}^{\xi_{i+1}}\int_{\eta_{j}}^{\eta_{j+1}}\phi_{\gamma}(\xi,\eta)(\xi^2+\eta^2)^{\frac{\gamma-2-\beta}{2}}d\eta d\xi-I_{i,j}\right) \\
        =&\uppercase\expandafter{\romannumeral1}+\uppercase\expandafter{\romannumeral2}.
\end{split}
\end{equation}
For the first part of (\ref{error1}), there exists
\begin{equation}
\begin{split}
|\uppercase\expandafter{\romannumeral1}|&\leq\int_{\xi_{0}}^{\xi_1}\int_{\eta_{0}}^{\eta_1}\left(\left|\phi_{\gamma}(\xi,\eta)\right|+\frac{k_{\gamma}}{4}\left|\phi_{\gamma}(\xi_0,\eta_1)+\phi_{\gamma}(\xi_1,\eta_0)+\phi_{\gamma}(\xi_1,\eta_1)\right|\right)(\xi^2+\eta^2)^{\frac{\gamma-2-\beta}{2}} d\eta d\xi \\
&\leq\int_{\xi_{0}}^{\xi_1}\int_{\eta_{0}}^{\eta_1}\left(C(\xi^2+\eta^2)^{1-\frac{\gamma}{2}}+Ch^{2-\gamma}\right)(\xi^2+\eta^2)^{\frac{\gamma-2-\beta}{2}} d\eta d\xi .
\end{split}
\end{equation}
Taking $\xi =ph$, $\eta =qh$, we have
\begin{equation}
\begin{split}
|\uppercase\expandafter{\romannumeral1}|\leq& Ch^{2-\beta}\int_{{0}}^{1}\int_{{0}}^{1}(p^2+q^2)^{-\frac{\beta}{2}}dq dp \\
                                          & +Ch^{2-\beta}\int_{{0}}^{1}\int_{{0}}^{1}(p^2+q^2)^{\frac{\gamma-2-\beta}{2}}dq dp .
\end{split}
\end{equation}
Since $\beta<\gamma\leq2$, we obtain $-\beta>-2$ and $\gamma-2-\beta>-2$. Then it holds
\begin{equation}
|\uppercase\expandafter{\romannumeral1}|\leq Ch^{2-\beta}.
\end{equation}
For the second part of (\ref{error1}), according to Lemma \ref{lemmaintepola}, we have
\begin{equation}
\begin{split}
|\uppercase\expandafter{\romannumeral2}|\leq C&\sum_{
        \begin{subarray}
         ~i=0;j=0;\\(i,j)\neq (0,0)
        \end{subarray}
        }^{i=N_i-1;j=N_j-1}\int_{\xi_{i}}^{\xi_{i+1}}\int_{\eta_{j}}^{\eta_{j+1}}\left(\left\|\frac{\partial^2\phi_{\gamma}}{\partial\xi^2}\right\|_{L_\infty}+\left\|\frac{\partial^2\phi_{\gamma}}{\partial\eta^2}\right\|_{L_\infty}\right)h^2(\xi^2+\eta^2)^{\frac{\gamma-2-\beta}{2}}d\eta d\xi .
\end{split}
\end{equation}
Denote $\Omega_{i,j}=[\xi_{i},\xi_{i+1}]\times[\eta_{j},\eta_{j+1}]$. According to Lemma \ref{lemfunc2error}, we have
\begin{equation}
\begin{split}
  \left\|\frac{\partial^2\phi_{\gamma}}{\partial\xi^2}\right\|_{L_\infty(\Omega_{i,j})}&\leq C\sup_{(\xi,\eta)\in\Omega_{i,j}}\left((\xi^2+\eta^2)^{-\frac{\gamma}{2}}+(\xi^2+\eta^2)^{\frac{1}{2}-\frac{\gamma}{2}}+(\xi^2+\eta^2)^{1-\frac{\gamma}{2}}\right),\\
  \left\|\frac{\partial^2\phi_{\gamma}}{\partial\eta^2}\right\|_{L_\infty(\Omega_{i,j})}&\leq C\sup_{(\xi,\eta)\in\Omega_{i,j}}\left((\xi^2+\eta^2)^{-\frac{\gamma}{2}}+(\xi^2+\eta^2)^{\frac{1}{2}-\frac{\gamma}{2}}+(\xi^2+\eta^2)^{1-\frac{\gamma}{2}}\right).
\end{split}
\end{equation}
For any $(\xi,\eta)\in\Omega_{i,j}$ \,$(i,j\geq0,(i,j)\neq(0,0))$, there exists a constant $C$ satisfying
\begin{equation}
\begin{split}
&\sup_{(\xi,\eta)\in\Omega_{i,j}}\left((\xi^2+\eta^2)^{-\frac{\gamma}{2}}\right)\leq C(\xi^2+\eta^2)^{-\frac{\gamma}{2}},\\
&\sup_{(\xi,\eta)\in\Omega_{i,j}}\left((\xi^2+\eta^2)^{\frac{1}{2}-\frac{\gamma}{2}}\right)\leq C(\xi^2+\eta^2)^{\frac{1}{2}-\frac{\gamma}{2}},\\
&\sup_{(\xi,\eta)\in\Omega_{i,j}}\left((\xi^2+\eta^2)^{1-\frac{\gamma}{2}}\right)\leq C(\xi^2+\eta^2)^{1-\frac{\gamma}{2}}.
\end{split}
\end{equation}
Thus
\begin{equation}
\begin{split}
        |\uppercase\expandafter{\romannumeral2}|\leq &C\sum_{
        \begin{subarray}
         ~i=0;j=0;\\(i,j)\neq (0,0)
        \end{subarray}
        }^{i=N_i-1;j=N_j-1}\int_{\xi_{i}}^{\xi_{i+1}}\int_{\eta_{j}}^{\eta_{j+1}}h^2\left((\xi^2+\eta^2)^{\frac{-2-\beta}{2}}+(\xi^2+\eta^2)^{\frac{-1-\beta}{2}}+(\xi^2+\eta^2)^{-\frac{\beta}{2}}\right)d\eta d\xi .\\
        \leq& |\uppercase\expandafter{\romannumeral2}_1|+|\uppercase\expandafter{\romannumeral2}_2|+|\uppercase\expandafter{\romannumeral2}_3|.\\
\end{split}
\end{equation}
Taking $\xi =ph$, $\eta =qh$, we have
\begin{equation}
\begin{split}
  |\uppercase\expandafter{\romannumeral2}_1|\leq& C\sum_{
        \begin{subarray}
         ~i=0;j=0;\\(i,j)\neq (0,0)
        \end{subarray}
        }^{i=N_i-1;j=N_j-1}h^{2-\beta}\int_{{i}}^{i+1}\int_{{j}}^{j+1}(p^2+q^2)^{\frac{-2-\beta}{2}}dq dp .
  \end{split}
\end{equation}
And since $-2-\beta<-2$, it holds
\begin{equation}\label{eques2_1}
  |\uppercase\expandafter{\romannumeral2}_1|\leq Ch^{2-\beta}.
\end{equation}
Then, we have
\begin{equation}
  \begin{split}
    |\uppercase\expandafter{\romannumeral2}_3|\leq Ch^2\sum_{
        \begin{subarray}
         ~i=0;j=0;\\(i,j)\neq (0,0)
        \end{subarray}
        }^{i=N_i-1;j=N_j-1}\int_{\xi_{i}}^{\xi_{i+1}}\int_{\eta_{j}}^{\eta_{j+1}}(\xi^2+\eta^2)^{-\frac{\beta}{2}}d\eta d\xi .\\
  \end{split}
\end{equation}
Take $\xi=r\cos(\theta)$, $\eta=r\sin(\theta)$. Since $0<\beta<2$, there exists
\begin{equation}\label{eques2_3}
  \begin{split}
    |\uppercase\expandafter{\romannumeral2}_3|\leq& Ch^2\int_{0}^{\frac{\pi}{2}}\int_{h}^{\sqrt{2}L}r^{1-\beta}dr d\theta\\
    \leq&Ch^2.
  \end{split}
\end{equation}
For $|\uppercase\expandafter{\romannumeral2}_2|$, being similar to $|\uppercase\expandafter{\romannumeral2}_1|$ and $|\uppercase\expandafter{\romannumeral2}_3|$, we have
\begin{equation}\label{eques2_2}
  |\uppercase\expandafter{\romannumeral2}_2|\leq\left\{
  \begin{split}
    &Ch^2~~~~~~~\beta\in(0,1];\\
    &Ch^{3-\beta}~~\beta\in(1,2).
  \end{split}\right.
\end{equation}
From (\ref{eques2_1}), (\ref{eques2_3}) and (\ref{eques2_2}), it can be obtained that
\begin{equation}
\begin{split}
|\uppercase\expandafter{\romannumeral2}|\leq &C h^{2-\beta}.
\end{split}
\end{equation}
So for $u(x,y)\in C^2(\mathbb{R}^2)$, we have
\begin{equation}
\left\|e^h_{\beta,\gamma}(x,y)\right\|_{L_\infty}\leq Ch^{2-\beta}.
\end{equation}
Then, the proof is completed.
\end{proof}

\section{Error estimates}
Now, we turn to the convergence proof of the designed scheme for the tempered fractional Poisson problem with Dirichlet boundary conditions (\ref{defequ1}).

%In this sections, we provide a convergence proof of solution to the tempered fractional Poisson problem with Dirichlet boundary conditions (\ref{defequ1}) by the present method.

\begin{lemma}\label{lemmaGersgorin}\cite{Axelsson1996}
  The spectrum $\lambda(A)$ of the matrix $A=[a_{i,j}]$ is enclosed in the union of the discs
  \begin{equation}
  C_i=\{z\in \mathbb{C};|z-a_{i,i}|\leq\sum_{i\neq j}|a_{i,j}|\},~1\leq i\leq n
  \end{equation}
  and in the union of the discs
  \begin{equation}
  C'_i=\{z\in \mathbb{C};|z-a_{i,i}|\leq\sum_{i\neq j}|a_{j,i}|\},~1\leq i\leq n.
  \end{equation}
\end{lemma}

Next, we give the proposition of the weights $w_{i,j}$. From (\ref{equweightoffl}), it's easy to verify the following properties of weights.
\begin{proposition}\label{proweight}
  The weights of the tempered fractional Laplacian satisfy %the following properties,
  \begin{equation}
  \left\{
  \begin{split}
    &\sum_{i=-N_i}^{i=N_i}\sum_{j=-N_j}^{j=N_j}w_{|i|,|j|}>CG^\infty>0;\\
    &w_{i,j}<0,~~~~~(i,j)~\neq (0,0).
  \end{split}
  \right.
  \end{equation}
\end{proposition}
\begin{proof}
According to (\ref{equweightoffl}), we just need to prove that $W^1_{i,j}$, $W^2_{i,j}$, $W^3_{i,j}$, $W^4_{i,j}>0$. Combining (\ref{equdefG}) and (\ref{equdefW}), there exists
\begin{equation}
\begin{split}
  W^1_{i,j}&=\frac{1}{h^2}\int_{\xi_{i}}^{\xi_{i+1}}\int_{\eta_{j}}^{\eta_{j+1}}(\xi-\xi_{i+1})(\eta-\eta_{i+1})(\xi^2+\eta^2)^{\frac{\gamma-2-\beta}{2}}d\eta d\xi\\
  &\geq 0.
  \end{split}
\end{equation}
The proof for $W^2_{i,j}$, $W^3_{i,j}$ and $W^4_{i,j}$ is similar to the one for $W^1_{i,j}$. Combining $G^\infty>0$ and $G_{0,0}>0$, one can get $\sum_{i=-N_i}^{i=N_i}\sum_{j=N_j}^{j=N_j}w_{|i|,|j|}>CG^\infty>0$ for some $C>0$. For $w_{i,j}<0\,((i,j)\neq (0,0))$, one can directly get from (\ref{equweightoffl}).
\end{proof}

According to Proposition \ref{proweight} and Lemma \ref{lemmaGersgorin}, the minimum eigenvalue of $B$ satisfies
\begin{equation}
\lambda_{min}(B)>CG^\infty>0.
\end{equation}
So $B$ is a strictly diagonally dominant and symmetric positive definite matrix.
\begin{theorem}\label{thmposerror}
Suppose that $u$ is the exact solution of the tempered fractional Poisson equation (\ref{defequ1}) and $\mathbf{U}_h$ is the solution of the finite difference scheme (\ref{matex1}). Then, there are
\begin{equation}
\begin{split}
   &\left\|\mathbf{U}-\mathbf{U}_h\right\|\leq C\left\|(\Delta+\lambda)_{h}^{\frac{\beta}{2}}\mathbf{U}-((\Delta+\lambda)_{h}^{\frac{\beta}{2}}\mathbf{U}_h)\right\|,\\
   & \left\|\mathbf{U}-\mathbf{U}_h\right\|_{\infty}\leq C\left\|(\Delta+\lambda)_{h}^{\frac{\beta}{2}}\mathbf{U}-((\Delta+\lambda)_{h}^{\frac{\beta}{2}}\mathbf{U}_h)\right\|_{\infty}.
\end{split}
\end{equation}
\end{theorem}
\begin{proof}
According to the definition of $G^{\infty}$, taking an inner product of (\ref{matex1}) with $\mathbf{U}_h$ and using the Cauchy-Schwarz inequality, we have
\begin{equation}
CG^{\infty}\left\|\mathbf{U}_h\right\|^2\leq(B\mathbf{U}_h,\mathbf{U}_h)\leq \left\|F\right\|\left\|\mathbf{U}_h\right\|,
\end{equation}
which leads to
\begin{equation}\label{equL2error10}
 \|\mathbf{U}_h\|^2\leq\frac{1}{CG^{\infty}}\|F\|\|\mathbf{U}_h\|.
\end{equation}
Thus
\begin{equation}\label{equL2error1}
 \|\mathbf{U}_h\|\leq\frac{1}{CG^{\infty}}\|F\|.
\end{equation}
Assuming $\|\mathbf{U}_h\|_\infty=|u^h_{p,q}|$, according to (\ref{equweightoffl}), we obtain that
\begin{equation}
\begin{split}
&u^h_{p,q}\left(\sum_{i=-N_i}^{i=N_i}\sum_{j=-N_j}^{j=N_j}w_{|i|,|j|}u^h_{p-i,q-j}-4c_{2,\beta,\lambda}G^\infty u^h_{p,q}\right)\\
=&u^h_{p,q}\left(\sum_{
        \begin{subarray}
         ~i=-N_i;j=-N_j;\\(i,j)\neq (0,0)
        \end{subarray}
        }^{i=N_i;j=N_j}w_{|i|,|j|}u^h_{p-i,q-j}+(w_{0,0}-4c_{2,\beta,\lambda}G^\infty) u^h_{p,q}\right)\\
\geq&\sum_{
        \begin{subarray}
         ~i=-N_i;j=-N_j;\\(i,j)\neq (0,0)
        \end{subarray}
        }^{i=N_i;j=N_j}-w_{|i|,|j|}((u^h_{p,q})^2-u^h_{p,q}u^h_{p-i,q-j})\\
\geq&0,
\end{split}
\end{equation}
which implies
\begin{equation}
CG^\infty \left\|\mathbf{U}_{h}\right\|_{\infty}\leq\left|F_{p,q}\right|.
\end{equation}
So we have
\begin{equation}\label{equinferror1}
CG^\infty \left\|\mathbf{U}_{h}\right\|_{\infty}\leq\|F\|_\infty.
\end{equation}
In addition, from (\ref{desdif2})
\begin{equation}\label{equdifdif}
\mathbf{B}(\mathbf{U}-\mathbf{U}_h)=(-(\Delta+\lambda)_{h}^{\frac{\beta}{2}}\mathbf{U})-(-(\Delta+\lambda)_{h}^{\frac{\beta}{2}}\mathbf{U}_h).
\end{equation}
Applying (\ref{equL2error1}) and (\ref{equinferror1}) to (\ref{equdifdif}), the desired results are obtained.
\end{proof}

\begin{theorem}
Suppose $u\in C^2({\mathbb{R}^2})$ is the exact solution of (\ref{defequ1}), and $\mathbf{U}_h$ is the solution of the difference scheme (\ref{matex1}). Then
\begin{equation}
  \left\|\mathbf{U}-\mathbf{U}_h\right\|\leq Ch^{2-\beta},~~\left\|\mathbf{U}-\mathbf{U}_h\right\|_{\infty}\leq Ch^{2-\beta}.
\end{equation}
\end{theorem}
\begin{proof}
Combining Theorem \ref{thmtrunct} and Theorem \ref{thmposerror} leads to that for $u\in C^2(\mathbb{R}^2)$,
\begin{equation}
 \left\|\mathbf{U}-\mathbf{U}_h\right\|\leq Ch^{2-\beta},~~\left\|\mathbf{U}-\mathbf{U}_h\right\|_{\infty}\leq Ch^{2-\beta}.
\end{equation}

\end{proof}

%%%%%%%%%%%%%%%%%%%%%%%%%%%%%%%%%%%%%%%%%%%%%%%%%%%%%%%%%%
\section{Numerical experiments}

In this section, extensive numerical experiments are performed, including verifying the theoretical results on convergence rates and showing the effectiveness of the scheme by simulating (\ref{defequ1}) without known solution. The convergence results for $\lambda=0$ are also reported. Without loss of generality, we consider the domain $\Omega=(-1,1)\times(-1,1)$.

%In this section, we verify the accuracy of our method in discretizing the tempered fractional Laplacian and solving the tempered fractional Poisson equation (\ref{defequ1}). And we also give the accuracy of the method for fractional Laplacian by setting $\lambda=0$. Without loss of generality, we consider the domain $\Omega=(-1,1)\times(-1,1)$.

\subsection{The truncation error of the tempered fractional Laplacian}
This subsection shows the truncation errors and convergence rates of discretizing the tempered fractional Laplacian. The $L_{\infty}$ norm and  $L_{2}$ norm are used to measure the truncation errors here.
\begin{example}
Compute $(\Delta+\lambda)^{\beta/2}u(x,y)$ with $u(x,y)=(1-x^2)^3(1-y^2)^3$ ($u(x,y)\in C^2(\mathbb{R}^2)$).

 Table \ref{tab:dC2r0g1a05} shows the accuracy of computing $(\Delta+\lambda)^{\beta/2}u(x,y)$ with  $\lambda=0$ and $\gamma=1+\frac{\beta}{2}$, which verifies the numerical discretizations for fractional Laplacian. Table \ref{tab:dC2r05g1a05} shows the accuracy of computing $(\Delta+\lambda)^{\beta/2}u(x,y)$ with $\lambda=0.5$ and $\gamma=1+\frac{\beta}{2}$. We find that for the fixed mesh size $h$, the numerical errors will be larger as the parameter $\beta$ increases and the truncation error is $O(h^{2-\beta})$ for any $\beta\in(0,2)$ from Table \ref{tab:dC2r0g1a05} and \ref{tab:dC2r05g1a05}. These results are consistent with the theoretical predictions.
% Table generated by Excel2LaTeX from sheet 'Sheet1'

\begin{table}[h]\fontsize{6pt}{10pt}\selectfont%生成浮动表格
 \begin{center}%\def\tabcolsep{28.5pt}%表格居中
  \caption {Numerical approximation errors and convergence orders for  $(\Delta+\lambda)^{\beta/2}(1-x^2)^3(1-y^2)^3$ with $\lambda=0$ and $\gamma=1+\frac{\beta}{2}$} \vspace{5pt}% 标题，离表格一定的距离
\begin{tabular*}{\linewidth}{@{\extracolsep{\fill}}*{10}{c}}                                    \hline  %画顶端的横线
 $\beta\backslash h$          &            &      1/8   &      1/16  &      1/32  &      1/64  &      1/128 &      1/256 \\
\hline
           &$L_\infty$            & 2.5155E-02 & 9.9477E-03 & 3.7440E-03 & 1.3771E-03 & 4.9991E-04 & 1.7998E-04 \\

       0.5 &Rate            &            &    1.3384  &    1.4098  &    1.4430  &    1.4619  &    1.4738  \\

           &$L_2$            & 1.5053E-02 & 6.0283E-03 & 2.2798E-03 & 8.4041E-04 & 3.0548E-04 & 1.1007E-04 \\

           &Rate            &            &    1.3202  &    1.4029  &    1.4397  &    1.4600  &    1.4727  \\

           &$L_\infty$            & 9.0498E-02 & 4.2795E-02 & 1.9321E-02 & 8.5654E-03 & 3.7652E-03 & 1.6479E-03 \\

       0.8 &Rate            &            &    1.0804  &    1.1473  &    1.1736  &    1.1858  &    1.1921  \\

           &$L_2$            & 5.4199E-02 & 2.5992E-02 & 1.1792E-02 & 5.2379E-03 & 2.3045E-03 & 1.0091E-03 \\

           &Rate            &            &    1.0602  &    1.1402  &    1.1708  &    1.1845  &    1.1914  \\

           &$L_\infty$            & 4.0132E-01 & 2.4528E-01 & 1.4377E-01 & 8.3182E-02 & 4.7909E-02 & 2.7548E-02 \\

       1.2 &Rate            &            &    0.7103  &    0.7706  &    0.7895  &    0.7960  &    0.7984  \\

           &$L_2$            & 2.4038E-01 & 1.4921E-01 & 8.7882E-02 & 5.0922E-02 & 2.9344E-02 & 1.6876E-02 \\

           &Rate            &            &    0.6880  &    0.7637  &    0.7873  &    0.7952  &    0.7981  \\

           &$L_\infty$            & 1.1249E+00 & 8.4067E-01 & 6.0367E-01 & 4.2877E-01 & 3.0360E-01 & 2.1477E-01 \\

       1.5 &Rate            &            &    0.4202  &    0.4778  &    0.4935  &    0.4981  &    0.4994  \\

           &$L_2$            & 6.7336E-01 & 5.1163E-01 & 3.6915E-01 & 2.6256E-01 & 1.8598E-01 & 1.3158E-01 \\

           &Rate            &            &    0.3963  &    0.4709  &    0.4916  &    0.4975  &    0.4992  \\
           \hline
    \end{tabular*}\label{tab:dC2r0g1a05}%\vspace{-15pt}
  \end{center}
\end{table}

 Comparing Table \ref{tab:dC2r0g1a05} with \ref{tab:dC2r05g1a05}, it's easy to see that the convergence rates are independent of $\lambda$  and the numerical errors become smaller as the parameter $\lambda$ increases for fixed $h$ and $\beta$.
\begin{table}[h]\fontsize{6pt}{10pt}\selectfont%生成浮动表格
 \begin{center}%\def\tabcolsep{28.5pt}%表格居中
  \caption {Numerical approximation errors and convergence orders for  $(\Delta+\lambda)^{\beta/2}(1-x^2)^3(1-y^2)^3$ with $\lambda=0.5$ and $\gamma=1+\frac{\beta}{2}$} \vspace{5pt}% 标题，离表格一定的距离
\begin{tabular*}{\linewidth}{@{\extracolsep{\fill}}*{10}{c}}                                    \hline  %画顶端的横线
  $\beta\backslash h$         &            &      1/8   &      1/16  &      1/32  &      1/64  &      1/128 &      1/256 \\
\hline
           &$L_\infty$             & 2.1316E-02 & 9.0524E-03 & 3.5350E-03 & 1.3277E-03 & 4.8808E-04 & 1.7711E-04 \\

       0.5 &Rate            &            &    1.2356  &    1.3566  &    1.4128  &    1.4437  &    1.4624  \\

           &$L_2$            & 1.2757E-02 & 5.4918E-03 & 2.1546E-03 & 8.1088E-04 & 2.9842E-04 & 1.0836E-04 \\

           &Rate            &            &    1.2159  &    1.3498  &    1.4099  &    1.4422  &    1.4615  \\

           &$L_\infty$             & 7.6879E-02 & 3.9268E-02 & 1.8428E-02 & 8.3409E-03 & 3.7089E-03 & 1.6338E-03 \\

       0.8 &Rate            &            &    0.9693  &    1.0915  &    1.1436  &    1.1692  &    1.1827  \\

           &$L_2$            & 4.6014E-02 & 2.3861E-02 & 1.1252E-02 & 5.1020E-03 & 2.2705E-03 & 1.0005E-03 \\

           &Rate            &            &    0.9474  &    1.0845  &    1.1410  &    1.1681  &    1.1822  \\

           &$L_\infty$             & 3.3838E-01 & 2.2543E-01 & 1.3772E-01 & 8.1362E-02 & 4.7368E-02 & 2.7388E-02 \\

       1.2 &Rate            &            &    0.5860  &    0.7110  &    0.7593  &    0.7805  &    0.7904  \\

           &$L_2$            & 2.0238E-01 & 1.3714E-01 & 8.4188E-02 & 4.9811E-02 & 2.9013E-02 & 1.6778E-02 \\

           &Rate            &            &    0.5615  &    0.7039  &    0.7572  &    0.7798  &    0.7901  \\

           &$L_\infty$             & 9.3441E-01 & 7.6900E-01 & 5.7753E-01 & 4.1937E-01 & 3.0023E-01 & 2.1357E-01 \\

       1.5 &Rate            &            &    0.2811  &    0.4131  &    0.4617  &    0.4821  &    0.4914  \\

           &$L_2$            & 5.5809E-01 & 4.6792E-01 & 3.5317E-01 & 2.5680E-01 & 1.8392E-01 & 1.3084E-01 \\

           &Rate            &            &    0.2542  &    0.4059  &    0.4597  &    0.4816  &    0.4912  \\
\hline
    \end{tabular*}\label{tab:dC2r05g1a05}%\vspace{-15pt}
  \end{center}
\end{table}

\begin{tabular}{cccccccc}
\hline

\end{tabular}

\end{example}

\subsection{Convergence rates for solving the tempered fractional Poisson equation}
\begin{example}
We solve (\ref{defequ1}) with different $\beta$ and $\gamma$, and the exact solution is taken as  $u(x,y)=(1-x^2)^3(1-y^2)^3$, where $u\in C^2(\mathbb{R}^2)$. The source term $f(x,y)$ is obtained numerically by the algorithm in Appendix A.
 %We consider that the exact solution of (\ref{defequ1}) is $u(x,y)=(1-x^2)^3(1-y^2)^3$ with different $\beta$ and $\gamma$, where $u\in C^2(\mathbb{R}^2)$. The source term $f(x,y)$ is obtained numerically by the algorithm in Appendix A. The error and convergence rates are shown as follows.  The $L_{\infty}$ norm and  $L_{2}$ norm are also used to measure the numerical error here.

Table \ref{tab:sC2r0g1a05} shows that the convergence rate is $O(h^{2-\beta})$ when $\lambda=0$ and $\gamma=1+\frac{\beta}{2}$. Table \ref{tab:sC2r05g1a05} shows that the convergence rate is also $O(h^{2-\beta})$ when $\lambda=0.5$ and $\gamma=1+\frac{\beta}{2}$. The results show that $\lambda$ has no effect on the convergence rates  when $\gamma=1+\frac{\beta}{2}$.
% Table generated by Excel2LaTeX from sheet 'Sheet1'
\begin{table}[h]\fontsize{6pt}{10pt}\selectfont%生成浮动表格
 \begin{center}%\def\tabcolsep{28.5pt}%表格居中
  \caption {Errors and convergence orders of $(\Delta+\lambda)^{\beta/2}(1-x^2)^3(1-y^2)^3=f$ with $\lambda=0$ and $\gamma=1+\frac{\beta}{2}$} \vspace{5pt}% 标题，离表格一定的距离
\begin{tabular*}{\linewidth}{@{\extracolsep{\fill}}*{10}{c}}                                    \hline  %画顶端的横线
$\beta\backslash h$&  &      1/8   &      1/16  &      1/32  &      1/64  &      1/128 &      1/256 \\
\hline
           &$L_\infty$            & 1.2775E-02 & 5.0666E-03 & 1.9088E-03 & 7.0207E-04 & 2.5478E-04 & 9.1692E-05 \\

       0.5 &Rate            &            &    1.3343  &    1.4083  &    1.4430  &    1.4624  &    1.4744  \\

           &$L_2$            & 7.4608E-03 & 2.9768E-03 & 1.1253E-03 & 4.1463E-04 & 1.5061E-04 & 5.4231E-05 \\

           &Rate            &            &    1.3256  &    1.4034  &    1.4405  &    1.4610  &    1.4736  \\

           &$L_\infty$            & 3.0819E-02 & 1.4692E-02 & 6.6674E-03 & 2.9629E-03 & 1.3037E-03 & 5.7077E-04 \\

       0.8 &Rate            &            &    1.0687  &    1.1399  &    1.1701  &    1.1844  &    1.1916  \\

           &$L_2$            & 1.8187E-02 & 8.6468E-03 & 3.9240E-03 & 1.7442E-03 & 7.6760E-04 & 3.3610E-04 \\

           &Rate            &            &    1.0727  &    1.1398  &    1.1698  &    1.1842  &    1.1915  \\

           &$L_\infty$            & 7.9801E-02 & 4.9731E-02 & 2.9660E-02 & 1.7359E-02 & 1.0070E-02 & 5.8144E-03 \\

       1.2 &Rate            &            &    0.6822  &    0.7456  &    0.7728  &    0.7857  &    0.7923  \\

           &$L_2$            & 4.9372E-02 & 3.0390E-02 & 1.8037E-02 & 1.0534E-02 & 6.1039E-03 & 3.5226E-03 \\

           &Rate            &            &    0.7001  &    0.7526  &    0.7759  &    0.7873  &    0.7931  \\

           &$L_\infty$            & 1.4604E-01 & 1.1180E-01 & 8.2591E-02 & 6.0044E-02 & 4.3266E-02 & 3.1000E-02 \\

       1.5 &Rate            &            &    0.3854  &    0.4369  &    0.4600  &    0.4728  &    0.4810  \\

           &$L_2$            & 9.4657E-02 & 7.1555E-02 & 5.2579E-02 & 3.8120E-02 & 2.7422E-02 & 1.9626E-02 \\

           &Rate            &            &    0.4037  &    0.4446  &    0.4639  &    0.4752  &    0.4826  \\
\hline
    \end{tabular*}\label{tab:sC2r0g1a05}%\vspace{-15pt}
  \end{center}
\end{table}

% Table generated by Excel2LaTeX from sheet 'Sheet1'
% Table generated by Excel2LaTeX from sheet 'Sheet1'
\begin{table}[h]\fontsize{6pt}{10pt}\selectfont%生成浮动表格
 \begin{center}%\def\tabcolsep{28.5pt}%表格居中
  \caption {Errors and convergence orders of $(\Delta+\lambda)^{\beta/2}(1-x^2)^3(1-y^2)^3=f$ with $\lambda=0.5$ and $\gamma=1+\frac{\beta}{2}$} \vspace{5pt}% 标题，离表格一定的距离
\begin{tabular*}{\linewidth}{@{\extracolsep{\fill}}*{10}{c}}                                    \hline  %画顶端的横线
$\beta\backslash h$         &            &      1/8   &      1/16  &      1/32  &      1/64  &      1/128 &      1/256 \\
\hline
           &$L_\infty$            & 2.4304E-02 & 1.0338E-02 & 4.0423E-03 & 1.5182E-03 & 5.5784E-04 & 2.0232E-04 \\

       0.5 &Rate            &            &    1.2332  &    1.3547  &    1.4128  &    1.4444  &    1.4632  \\

           &$L_2$           & 1.4618E-02 & 6.1735E-03 & 2.4088E-03 & 9.0392E-04 & 3.3199E-04 & 1.2037E-04 \\

           &Rate            &            &    1.2436  &    1.3578  &    1.4140  &    1.4451  &    1.4636  \\

           &$L_\infty$            & 4.5025E-02 & 2.3185E-02 & 1.0960E-02 & 4.9802E-03 & 2.2183E-03 & 9.7786E-04 \\

       0.8 &Rate            &            &    0.9576  &    1.0809  &    1.1380  &    1.1667  &    1.1818  \\

           &$L_2$            & 2.7775E-02 & 1.4132E-02 & 6.6526E-03 & 3.0174E-03 & 1.3429E-03 & 5.9175E-04 \\

           &Rate            &            &    0.9748  &    1.0870  &    1.1406  &    1.1679  &    1.1823  \\

           &$L_\infty$            & 9.4091E-02 & 6.3722E-02 & 3.9735E-02 & 2.3827E-02 & 1.4004E-02 & 8.1432E-03 \\

       1.2 &Rate            &            &    0.5623  &    0.6814  &    0.7378  &    0.7668  &    0.7821  \\

           &$L_2$            & 6.0576E-02 & 4.0453E-02 & 2.5086E-02 & 1.5004E-02 & 8.8063E-03 & 5.1173E-03 \\

           &Rate            &            &    0.5825  &    0.6894  &    0.7415  &    0.7687  &    0.7832  \\

           &$L_\infty$            & 1.5112E-01 & 1.2580E-01 & 9.7124E-02 & 7.2367E-02 & 5.2879E-02 & 3.8195E-02 \\

       1.5 &Rate            &            &    0.2645  &    0.3732  &    0.4245  &    0.4526  &    0.4693  \\

           &$L_2$            & 1.0043E-01 & 8.2556E-02 & 6.3448E-02 & 4.7169E-02 & 3.4420E-02 & 2.4840E-02 \\

           &Rate            &            &    0.2827  &    0.3798  &    0.4277  &    0.4546  &    0.4706  \\
\hline
    \end{tabular*}\label{tab:sC2r05g1a05}%\vspace{-15pt}
  \end{center}
\end{table}

But when choosing $\lambda=0$ and $\gamma=2$, the convergence rates showed in Table \ref{tab:sC2r0g2} are higher than theoretical convergence rates; for any $\beta\in (0,2)$, the convergence rate is $O(h^2)$. For $\lambda>0$, the convergence rates showed in Table \ref{tab:sC2r05g2} depend on $\beta$; when $\beta<1$, the convergence rate is $O(h^2)$, and $\beta>1$ the convergence rate is $O(h^{3-\beta})$. This phenomenon indicates that the provided scheme works very well for the equation (\ref{defequ1}) when $\gamma=2$.
% Table generated by Excel2LaTeX from sheet 'Sheet1'
\begin{table}[h]\fontsize{6pt}{10pt}\selectfont%生成浮动表格
 \begin{center}%\def\tabcolsep{28.5pt}%表格居中
  \caption {Errors and convergence orders of $(\Delta+\lambda)^{\beta/2}(1-x^2)^3(1-y^2)^3=f$ with $\lambda=0$ and $\gamma=2$} \vspace{5pt}% 标题，离表格一定的距离
\begin{tabular*}{\linewidth}{@{\extracolsep{\fill}}*{10}{c}}                                    \hline  %画顶端的横线
$\beta\backslash h$&            &      1/8   &      1/16  &      1/32  &      1/64  &      1/128 &      1/256 \\
\hline
           &$L_\infty$            & 9.1029E-04 & 1.9637E-04 & 4.6351E-05 & 1.1347E-05 & 2.8156E-06 & 7.0180E-07 \\

       0.5 &Rate            &            &    2.2128  &    2.0829  &    2.0302  &    2.0109  &    2.0043  \\

           &$L_2$            & 5.7429E-04 & 1.2021E-04 & 2.7950E-05 & 6.8027E-06 & 1.6843E-06 & 4.1961E-07 \\

           &Rate            &            &    2.2562  &    2.1046  &    2.0387  &    2.0140  &    2.0050  \\

           &$L_\infty$            & 1.8685E-03 & 3.9361E-04 & 9.0560E-05 & 2.1800E-05 & 5.3592E-06 & 1.3295E-06 \\

       0.8 &Rate            &            &    2.2471  &    2.1198  &    2.0545  &    2.0242  &    2.0111  \\

           &$L_2$            & 1.2028E-03 & 2.4720E-04 & 5.5768E-05 & 1.3284E-05 & 3.2495E-06 & 8.0453E-07 \\

           &Rate           &            &    2.2827  &    2.1481  &    2.0697  &    2.0314  &    2.0140  \\

           &$L_\infty$            & 3.8160E-03 & 7.8100E-04 & 1.7136E-04 & 3.9489E-05 & 9.3971E-06 & 2.2808E-06 \\

       1.2 &Rate            &            &    2.2887  &    2.1883  &    2.1175  &    2.0712  &    2.0427  \\

           &$L_2$            & 2.4861E-03 & 5.0566E-04 & 1.0904E-04 & 2.4736E-05 & 5.8223E-06 & 1.4037E-06 \\

           &Rate            &            &    2.2977  &    2.2133  &    2.1401  &    2.0870  &    2.0523  \\

           &$L_\infty$            & 6.3257E-03 & 1.3143E-03 & 2.8396E-04 & 6.3320E-05 & 1.4493E-05 & 3.3880E-06 \\

       1.5 &Rate            &            &    2.2670  &    2.2105  &    2.1649  &    2.1273  &    2.0968  \\

           &$L_2$            & 4.0732E-03 & 8.5417E-04 & 1.8368E-04 & 4.0561E-05 & 9.1888E-06 & 2.1290E-06 \\

           &Rate            &            &    2.2536  &    2.2174  &    2.1790  &    2.1422  &    2.1097  \\
\hline
    \end{tabular*}\label{tab:sC2r0g2}%\vspace{-15pt}
  \end{center}
\end{table}

% Table generated by Excel2LaTeX from sheet 'Sheet1'
\begin{table}[h]\fontsize{6pt}{10pt}\selectfont%生成浮动表格
 \begin{center}%\def\tabcolsep{28.5pt}%表格居中
  \caption {Errors and convergence orders of $(\Delta+\lambda)^{\beta/2}(1-x^2)^3(1-y^2)^3=f$ with $\lambda=0.5$ and $\gamma=2$} \vspace{5pt}% 标题，离表格一定的距离
\begin{tabular*}{\linewidth}{@{\extracolsep{\fill}}*{10}{c}}                                    \hline  %画顶端的横线
$\beta\backslash h$&            &      1/8   &      1/16  &      1/32  &      1/64  &      1/128 &      1/256 \\
\hline
           &$L_\infty$            & 3.2465E-03 & 6.9337E-04 & 1.5697E-04 & 3.6753E-05 & 8.7832E-06 & 2.1265E-06 \\

       0.5 &Rate            &            &    2.2272  &    2.1431  &    2.0946  &    2.0650  &    2.0463  \\

           &$L_2$             & 1.9408E-03 & 4.1294E-04 & 9.3343E-05 & 2.1844E-05 & 5.2193E-06 & 1.2634E-06 \\

           & Rate           &            &    2.2326  &    2.1453  &    2.0953  &    2.0653  &    2.0466  \\

           &$L_\infty$            & 6.3390E-03 & 1.4709E-03 & 3.5361E-04 & 8.6471E-05 & 2.1319E-05 & 5.2764E-06 \\

       0.8 &Rate            &            &    2.1076  &    2.0564  &    2.0319  &    2.0201  &    2.0145  \\

           &$L_2$             & 3.7839E-03 & 8.7736E-04 & 2.1096E-04 & 5.1602E-05 & 1.2724E-05 & 3.1489E-06 \\

           &Rate            &            &    2.1086  &    2.0562  &    2.0315  &    2.0199  &    2.0146  \\

           &$L_\infty$            & 1.4822E-02 & 4.0632E-03 & 1.1494E-03 & 3.2981E-04 & 9.5204E-05 & 2.7540E-05 \\

       1.2 &Rate            &            &    1.8670  &    1.8217  &    1.8012  &    1.7925  &    1.7895  \\

           &$L_2$             & 8.9250E-03 & 2.4663E-03 & 7.0290E-04 & 2.0284E-04 & 5.8789E-05 & 1.7056E-05 \\

           &Rate            &            &    1.8555  &    1.8110  &    1.7930  &    1.7867  &    1.7853  \\

           &$L_\infty$            & 2.8458E-02 & 9.2196E-03 & 3.1122E-03 & 1.0731E-03 & 3.7417E-04 & 1.3125E-04 \\

       1.5 &Rate            &            &    1.6261  &    1.5668  &    1.5361  &    1.5200  &    1.5114  \\

           &$L_2$             & 1.7457E-02 & 5.7510E-03 & 1.9678E-03 & 6.8487E-04 & 2.4023E-04 & 8.4579E-05 \\

           &Rate            &            &    1.6019  &    1.5472  &    1.5227  &    1.5114  &    1.5060  \\
\hline
    \end{tabular*}\label{tab:sC2r05g2}%\vspace{-15pt}
  \end{center}
\end{table}

 Next, we give Figures \ref{fig:a05r0} and \ref{fig:a05r05} to show the influence of different $\gamma$ on the convergence rates. Figure \ref{fig:a05r0} shows that the convergence rate is almost $O(h^{2-\beta})$ except $\gamma=2$  when $\beta=0.5$ and $\lambda=0$; for the same mesh size $h$, the numerical errors become smaller as the parameter $\gamma$ increases. We can get the same results from Figure \ref{fig:a05r05} when $\beta=0.5$ and $\lambda=0.5$. Comparing Figure \ref{fig:a05r0} with \ref{fig:a05r05}, it's easy to note that $\gamma$ has the same influence on the convergence rates for any $\lambda$.
\begin{figure}[h]
  \begin{center}
  % Requires \usepackage{graphicx}
  \caption{$L_2$ errors and convergence orders for the system with different $\gamma$ when $\beta=0.5$ and $\lambda=0$}\label{fig:a05r0}
  \includegraphics[width=13.66cm,height=6cm,angle=0]{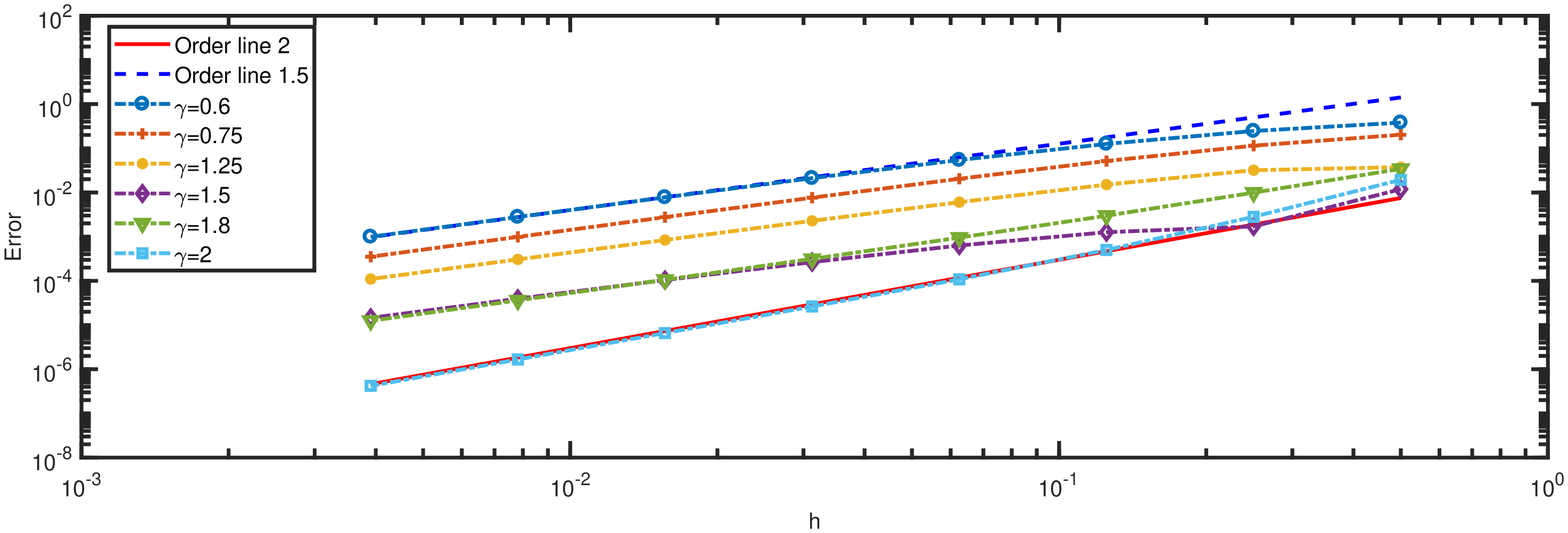}
  \end{center}
\end{figure}
\begin{figure}[h]
  \begin{center}
  % Requires \usepackage{graphicx}
  \caption{$L_2$ errors and convergence orders for the system with different $\gamma$ when $\beta=0.5$ and $\lambda=0.5$}\label{fig:a05r05}
  \includegraphics[width=13.66cm,height=6cm,angle=0]{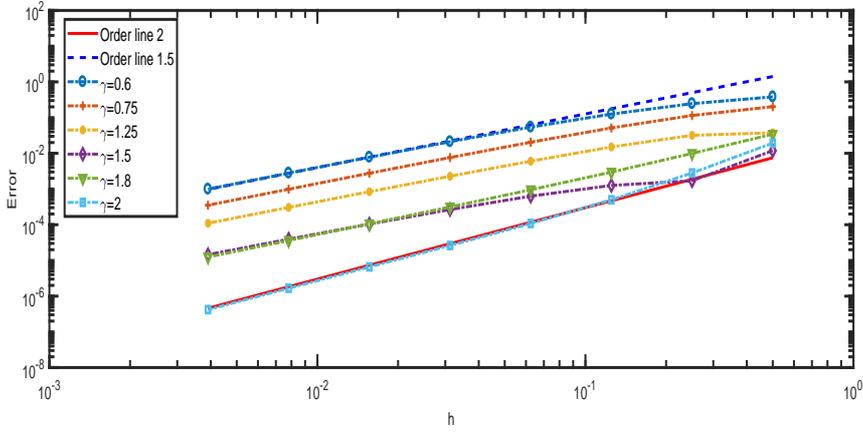}
  \end{center}
\end{figure}
\end{example}

Afterwards, we solve (\ref{defequ1}) with the exact solution $u=(1-x^2)^2(1-y^2)^2$, which has a lower regularity than one in Example 2.
\begin{example}
Taking the exact solution $u=(1-x^2)^2(1-y^2)^2$, the convergence rates are shown in Tables \ref{tab:sC1r0g1a05} and \ref{tab:sC1r05g1a05} with different $\beta$ and $\lambda$. It is easy to check that $u$ and $Du$ are continuous in $\mathbb{R}^2$, but $\frac{\partial^2 u}{\partial x^2}$ and $\frac{\partial^2 u}{\partial y^2}$ are discontinuous at the boundary of $\Omega$, so $u\in C^1(\mathbb{R}^2)$ and the second derivatives of $u$ are bounded. It can be noted that the provided scheme has the same convergence rates for $u\in C^2(\mathbb{R}^2)$ and $C^1(\mathbb{R}^2)$ with second bounded derivatives.
%We find it also has the same convergence rates to  $u\in C^2(\mathbb{R}^2)$, so our method is the same with a function which is in $C^1(\mathbb{R}^2)$ with second derivatives bounded.
\begin{table}[h]\fontsize{6pt}{10pt}\selectfont%生成浮动表格
 \begin{center}%\def\tabcolsep{28.5pt}%表格居中
  \caption {Errors and convergence orders of $(\Delta+\lambda)^{\beta/2}(1-x^2)^2(1-y^2)^2=f$ with $\lambda=0$ and $\gamma=1+\frac{\beta}{2}$} \vspace{5pt}% 标题，离表格一定的距离
\begin{tabular*}{\linewidth}{@{\extracolsep{\fill}}*{10}{c}}                                    \hline  %画顶端的横线
 $\beta\backslash h$          &            &      1/8   &      1/16  &      1/32  &      1/64  &      1/128 &      1/256 \\
\hline
           &$L_\infty$  & 1.0158E-02 & 3.9349E-03 & 1.4657E-03 & 5.3567E-04 & 1.9365E-04 & 6.9522E-05 \\

       0.5 &Rate        &            &    1.3682  &    1.4247  &    1.4522  &    1.4679  &    1.4779  \\

           &$L_2$       & 7.4130E-03 & 2.9212E-03 & 1.1009E-03 & 4.0553E-04 & 1.4735E-04 & 5.3072E-05 \\

           &Rate        &            &    1.3435  &    1.4078  &    1.4409  &    1.4606  &    1.4732  \\

           &$L_\infty$  & 2.6110E-02 & 1.2198E-02 & 5.4804E-03 & 2.4227E-03 & 1.0630E-03 & 4.6469E-04 \\

       0.8 &Rate        &            &    1.0980  &    1.1543  &    1.1777  &    1.1885  &    1.1938  \\

           &$L_2$       & 1.9157E-02 & 8.9819E-03 & 4.0460E-03 & 1.7920E-03 & 7.8730E-04 & 3.4443E-04 \\

           &Rate        &            &    1.0928  &    1.1505  &    1.1749  &    1.1866  &    1.1927  \\

           &$L_\infty$  & 7.2876E-02 & 4.4876E-02 & 2.6588E-02 & 1.5503E-02 & 8.9735E-03 & 5.1752E-03 \\

       1.2 &Rate        &            &    0.6995  &    0.7552  &    0.7782  &    0.7888  &    0.7940  \\

           &$L_2$       & 5.5491E-02 & 3.4120E-02 & 2.0173E-02 & 1.1745E-02 & 6.7921E-03 & 3.9151E-03 \\

           &Rate        &            &    0.7016  &    0.7582  &    0.7804  &    0.7901  &    0.7948  \\

           &$L_\infty$  & 1.3927E-01 & 1.0611E-01 & 7.8128E-02 & 5.6668E-02 & 4.0767E-02 & 2.9176E-02 \\

       1.5 &Rate        &            &    0.3923  &    0.4417  &    0.4633  &    0.4751  &    0.4826  \\

           &$L_2$       & 1.0921E-01 & 8.3270E-02 & 6.1236E-02 & 4.4359E-02 & 3.1880E-02 & 2.2799E-02 \\

           &Rate        &            &    0.3912  &    0.4434  &    0.4651  &    0.4766  &    0.4837  \\
\hline
\end{tabular*}\label{tab:sC1r0g1a05}%\vspace{-15pt}
  \end{center}
\end{table}

\begin{table}[h]\fontsize{6pt}{10pt}\selectfont%生成浮动表格
 \begin{center}%\def\tabcolsep{28.5pt}%表格居中
  \caption {Errors and convergence orders of $(\Delta+\lambda)^{\beta/2}(1-x^2)^2(1-y^2)^2=f$ with $\lambda=0.5$ and $\gamma=1+\frac{\beta}{2}$} \vspace{5pt}% 标题，离表格一定的距离
\begin{tabular*}{\linewidth}{@{\extracolsep{\fill}}*{10}{c}}                                    \hline  %画顶端的横线

  $\beta\backslash h$         &            &      1/8   &      1/16  &      1/32  &      1/64  &      1/128 &      1/256 \\
\hline
           &$L_\infty$            & 2.1303E-02 & 8.9076E-03 & 3.4504E-03 & 1.2888E-03 & 4.7199E-04 & 1.7083E-04 \\

       0.5 &Rate            &            &    1.2579  &    1.3683  &    1.4207  &    1.4492  &    1.4662  \\

           &$L_2$            & 1.5991E-02 & 6.6931E-03 & 2.5932E-03 & 9.6895E-04 & 3.5498E-04 & 1.2852E-04 \\

           &Rate            &            &    1.2565  &    1.3679  &    1.4202  &    1.4487  &    1.4657  \\

           &$L_\infty$            & 4.0896E-02 & 2.0785E-02 & 9.7538E-03 & 4.4132E-03 & 1.9611E-03 & 8.6335E-04 \\

       0.8 &Rate            &            &    0.9764  &    1.0915  &    1.1441  &    1.1702  &    1.1837  \\

           &$L_2$            & 3.1237E-02 & 1.5861E-02 & 7.4294E-03 & 3.3575E-03 & 1.4910E-03 & 6.5615E-04 \\

           &Rate            &            &    0.9778  &    1.0941  &    1.1459  &    1.1711  &    1.1842  \\

           &$L_\infty$            & 8.9126E-02 & 5.9987E-02 & 3.7250E-02 & 2.2277E-02 & 1.3072E-02 & 7.5938E-03 \\

       1.2 &Rate            &            &    0.5712  &    0.6874  &    0.7417  &    0.7692  &    0.7835  \\

           &$L_2$            & 6.9802E-02 & 4.7008E-02 & 2.9145E-02 & 1.7406E-02 & 1.0203E-02 & 5.9241E-03 \\

           &Rate            &            &    0.5704  &    0.6897  &    0.7437  &    0.7705  &    0.7844  \\

           &$L_\infty$            & 1.4667E-01 & 1.2195E-01 & 9.3995E-02 & 6.9937E-02 & 5.1048E-02 & 3.6843E-02 \\

       1.5 &Rate            &            &    0.2663  &    0.3756  &    0.4265  &    0.4542  &    0.4705  \\

           &$L_2$            & 1.1665E-01 & 9.7273E-02 & 7.4957E-02 & 5.5732E-02 & 4.0654E-02 & 2.9327E-02 \\

           &Rate            &            &    0.2621  &    0.3760  &    0.4276  &    0.4551  &    0.4712  \\
\hline
\end{tabular*}\label{tab:sC1r05g1a05}%\vspace{-15pt}
  \end{center}
\end{table}

\end{example}

Finally, we use the provided scheme to solve (\ref{defequ1}) with smooth right hand term.
\begin{example}\label{example4}
We consider the model (\ref{defequ1}) in $\Omega$ with the source term $f=1$. Here 
\begin{equation}
{\rm rate}=\frac{\ln(e_{2h}/{e_h})}{\ln(2)}
\end{equation}
is utilized to measure the convergence rates, where $u_h$ means the numerical solution under mesh size $h$ and $e_h=\|u_{2h}-u_{h}\|$.

Tables $\ref{tab:F1r0g1a05}$ and $\ref{tab:F1r05g2}$ show the numerical errors and the convergence rates with $\lambda=0$, $\gamma=1+\frac{\beta}{2}$ and $\lambda=0.5$, $\gamma=2$, respectively. The convergence rates are lower than desired ones because of the regularity of the exact solution $u$. 
%This phenomenon may be caused by the regularity of exact solution $u$.
These results are similar to the ones in one dimension \cite{Zhang2017}.
\begin{table}[h]\fontsize{6pt}{10pt}\selectfont%生成浮动表格
 \begin{center}%\def\tabcolsep{28.5pt}%表格居中
  \caption {Errors and convergence orders of $(\Delta+\lambda)^{\beta/2}u=1$ with $\lambda=0$ and $\gamma=1+\frac{\beta}{2}$} \vspace{5pt}% 标题，离表格一定的距离
\begin{tabular*}{\linewidth}{@{\extracolsep{\fill}}*{10}{c}}
\hline
  $\beta$         &            &   1/16-1/8 &  1/32-1/16 &  1/64-1/32 & 1/128-1/64 & 1/256-1/128 \\
\hline
           &$L_\infty$            & 5.5241E-02 & 4.2765E-02 & 3.4306E-02 & 2.8808E-02 & 2.4210E-02 \\

       0.5 &Rate            &            &    0.3693  &    0.3180  &    0.2520  &    0.2509  \\

           &$L_2$            & 2.8760E-02 & 1.7970E-02 & 1.0996E-02 & 6.6540E-03 & 4.0008E-03 \\

           &Rate            &            &    0.6784  &    0.7086  &    0.7247  &    0.7339  \\

           &$L_\infty$            & 4.1983E-02 & 3.1363E-02 & 2.3612E-02 & 1.7837E-02 & 1.3496E-02 \\

       0.8 &Rate            &            &    0.4207  &    0.4095  &    0.4046  &    0.4023  \\

           &$L_2$            & 2.8257E-02 & 1.6889E-02 & 9.7784E-03 & 5.5543E-03 & 3.1158E-03 \\

           &Rate            &            &    0.7425  &    0.7884  &    0.8160  &    0.8340  \\

           &$L_\infty$            & 2.5049E-02 & 1.5905E-02 & 1.0213E-02 & 6.6579E-03 & 4.3582E-03 \\

       1.2 &Rate            &            &    0.6553  &    0.6391  &    0.6173  &    0.6113  \\

           &$L_2$            & 2.2415E-02 & 1.4148E-02 & 8.6240E-03 & 5.1498E-03 & 3.0366E-03 \\

           &Rate            &            &    0.6639  &    0.7142  &    0.7438  &    0.7621  \\

           &$L_\infty$            & 2.1610E-02 & 1.5623E-02 & 1.1253E-02 & 8.0821E-03 & 5.7890E-03 \\

       1.5 &Rate            &            &    0.4680  &    0.4733  &    0.4775  &    0.4814  \\

           &$L_2$            & 1.5619E-02 & 1.1405E-02 & 8.2246E-03 & 5.8995E-03 & 4.2199E-03 \\

           &Rate            &            &    0.4536  &    0.4717  &    0.4793  &    0.4834  \\
\hline
\end{tabular*}\label{tab:F1r0g1a05}%\vspace{-15pt}
  \end{center}
\end{table}

\begin{table}[h]\fontsize{6pt}{10pt}\selectfont%生成浮动表格
 \begin{center}%\def\tabcolsep{28.5pt}%表格居中
  \caption {The error of $(\Delta+\lambda)^{\beta/2}u=1$ with $\lambda=0.5$ and $\gamma=2$} \vspace{5pt}% 标题，离表格一定的距离
\begin{tabular*}{\linewidth}{@{\extracolsep{\fill}}*{10}{c}}  
\hline
  $\beta$         &            &   1/16-1/8 &  1/32-1/16 &  1/64-1/32 & 1/128-1/64 & 1/256-1/128 \\
\hline
           &$L_\infty$            & 1.2879E-01 & 9.6038E-02 & 7.2802E-02 & 5.6334E-02 & 4.4445E-02 \\

       0.5 &Rate            &            &    0.4233  &    0.3996  &    0.3700  &    0.3420  \\

           &$L_2$            & 8.7742E-02 & 5.2782E-02 & 3.0383E-02 & 1.7117E-02 & 9.5561E-03 \\

           &Rate            &            &    0.7332  &    0.7968  &    0.8279  &    0.8409  \\

           &$L_\infty$            & 6.2353E-02 & 4.5041E-02 & 3.2535E-02 & 2.3712E-02 & 1.7459E-02 \\

       0.8 &Rate            &            &    0.4692  &    0.4693  &    0.4563  &    0.4417  \\

           &$L_2$            & 4.3717E-02 & 2.6380E-02 & 1.5039E-02 & 8.3099E-03 & 4.5118E-03 \\

           &Rate            &            &    0.7288  &    0.8107  &    0.8558  &    0.8811  \\

           &$L_\infty$            & 2.1830E-02 & 1.5132E-02 & 1.0074E-02 & 6.6261E-03 & 4.3480E-03 \\

       1.2 &Rate            &            &    0.5287  &    0.5870  &    0.6044  &    0.6078  \\

           &$L_2$            & 1.5544E-02 & 9.7421E-03 & 5.6015E-03 & 3.0677E-03 & 1.6317E-03 \\

           &Rate            &            &    0.6740  &    0.7984  &    0.8687  &    0.9108  \\

           &$L_\infty$            & 7.1398E-03 & 5.2215E-03 & 3.3499E-03 & 2.0513E-03 & 1.2336E-03 \\

       1.5 &Rate            &            &    0.4514  &    0.6403  &    0.7076  &    0.7336  \\

           &$L_2$            & 4.2726E-03 & 3.1385E-03 & 1.9557E-03 & 1.1211E-03 & 6.1242E-04 \\

           &Rate            &            &    0.4450  &    0.6824  &    0.8028  &    0.8723  \\
\hline
\end{tabular*}\label{tab:F1r05g2}%\vspace{-15pt}
  \end{center}
\end{table}

In statistical physics \cite{Deng:17-2}, the solution $u$ of Example \ref{example4} represents the mean first exit time of a particle starting at $(x,y)$ away from given domain $\Omega$. Figure $\ref{figDepen}$ shows the dynamical behaviors when $\lambda=0,~0.5$ and $\beta=0.5,~0.8,~1.2,~1.5$; for any $\lambda$ and $\beta$, the mean first exit times of particles starting near the center are longer than the particles starting near the boundary of $\Omega$; for any fixed $\lambda$, the mean first exit time is shorter as $\beta$ increases; when exponentially tempering the isotropic power law measure of the jump length, the mean first exit time of any fixed starting point is longer than before.
  \begin{figure}
  \centering
  % Requires \usepackage{graphicx}
  \subfigure[$\beta=0.5$, $\lambda=0$]{
    \begin{minipage}{3.5cm}
    \centering %子图居中
    \includegraphics[width=3.5cm]{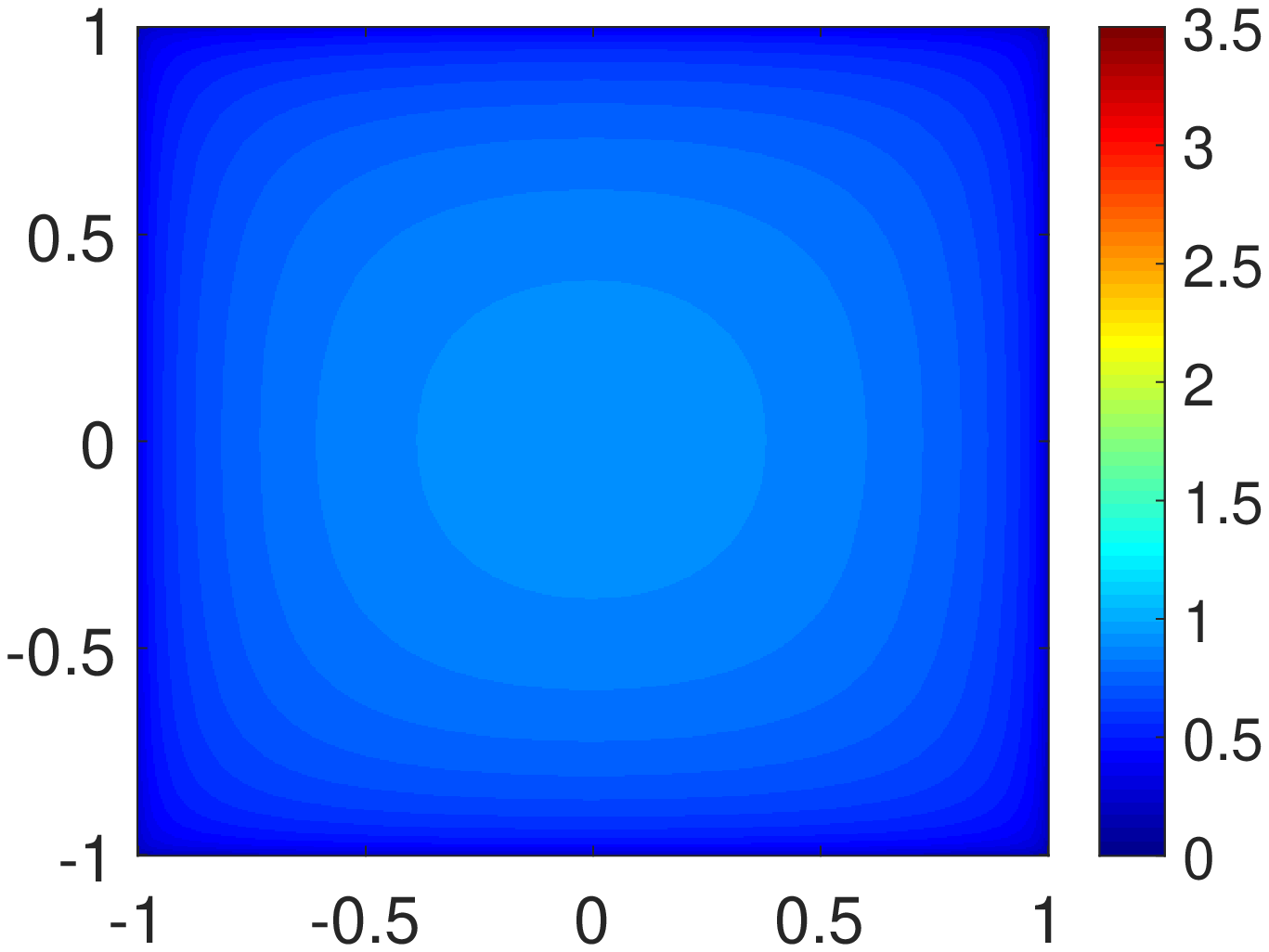}
    \end{minipage}
    }
    \subfigure[$\beta=0.8$, $\lambda=0$]{
    \begin{minipage}{3.5cm}
    \centering %子图居中
    \includegraphics[width=3.5cm]{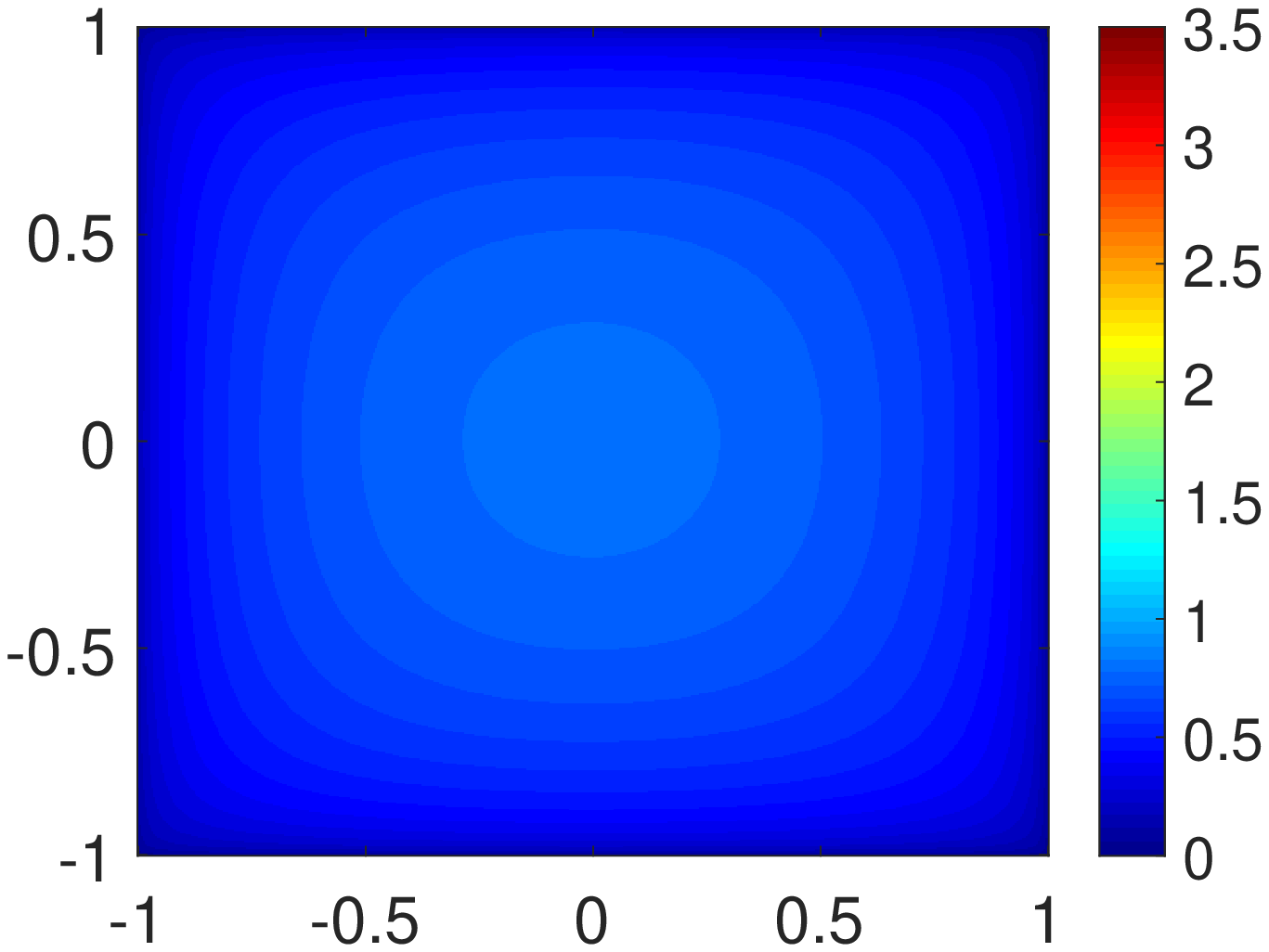}
    \end{minipage}
    }
      \subfigure[$\beta=1.2$, $\lambda=0$]{
    \begin{minipage}{3.5cm}
    \centering %子图居中
    \includegraphics[width=3.5cm]{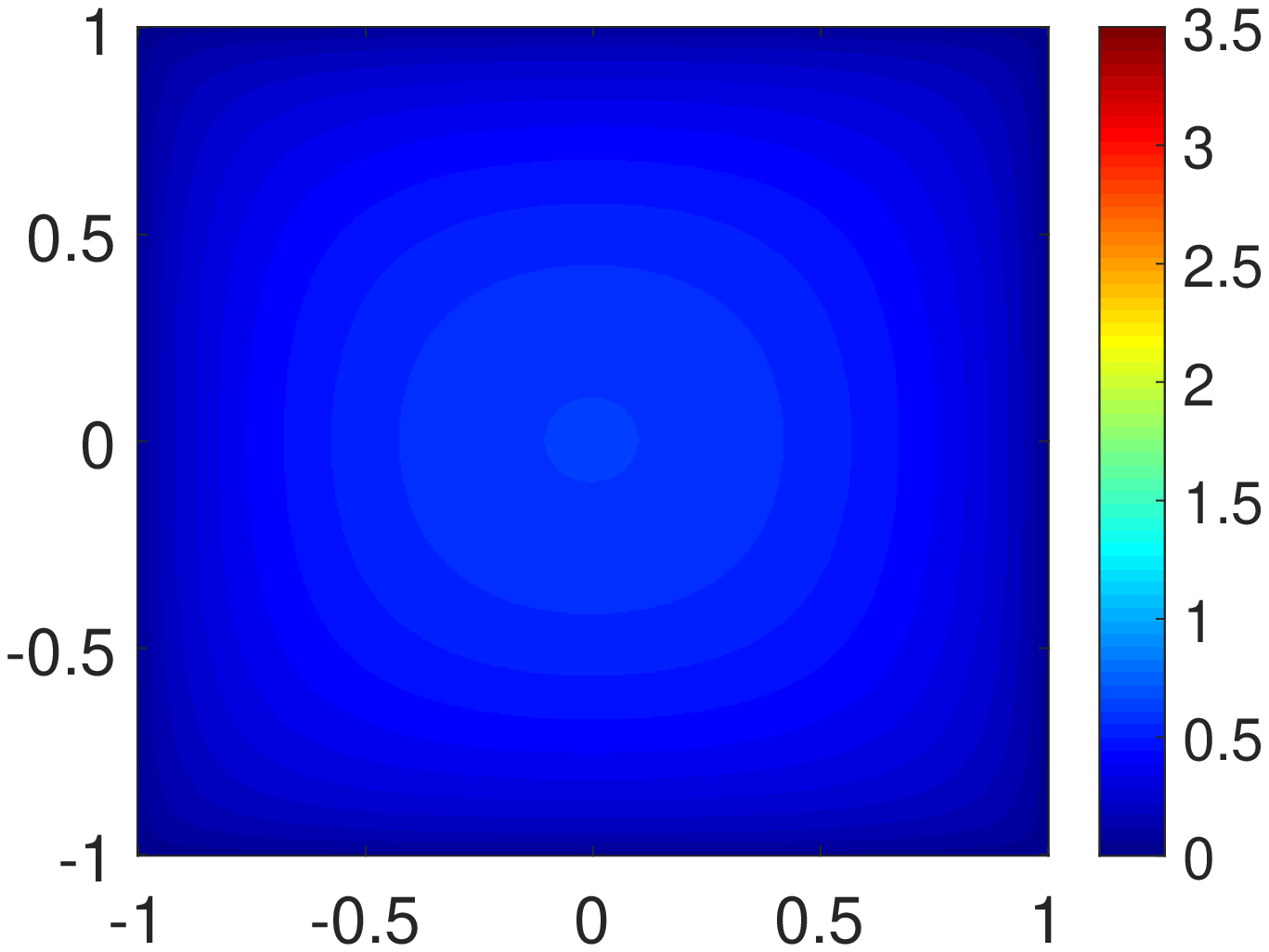}
    \end{minipage}
    }

     \subfigure[$\beta=0.5$, $\lambda=0.5$]{
    \begin{minipage}{3.5cm}
    \centering %子图居中
    \includegraphics[width=3.5cm]{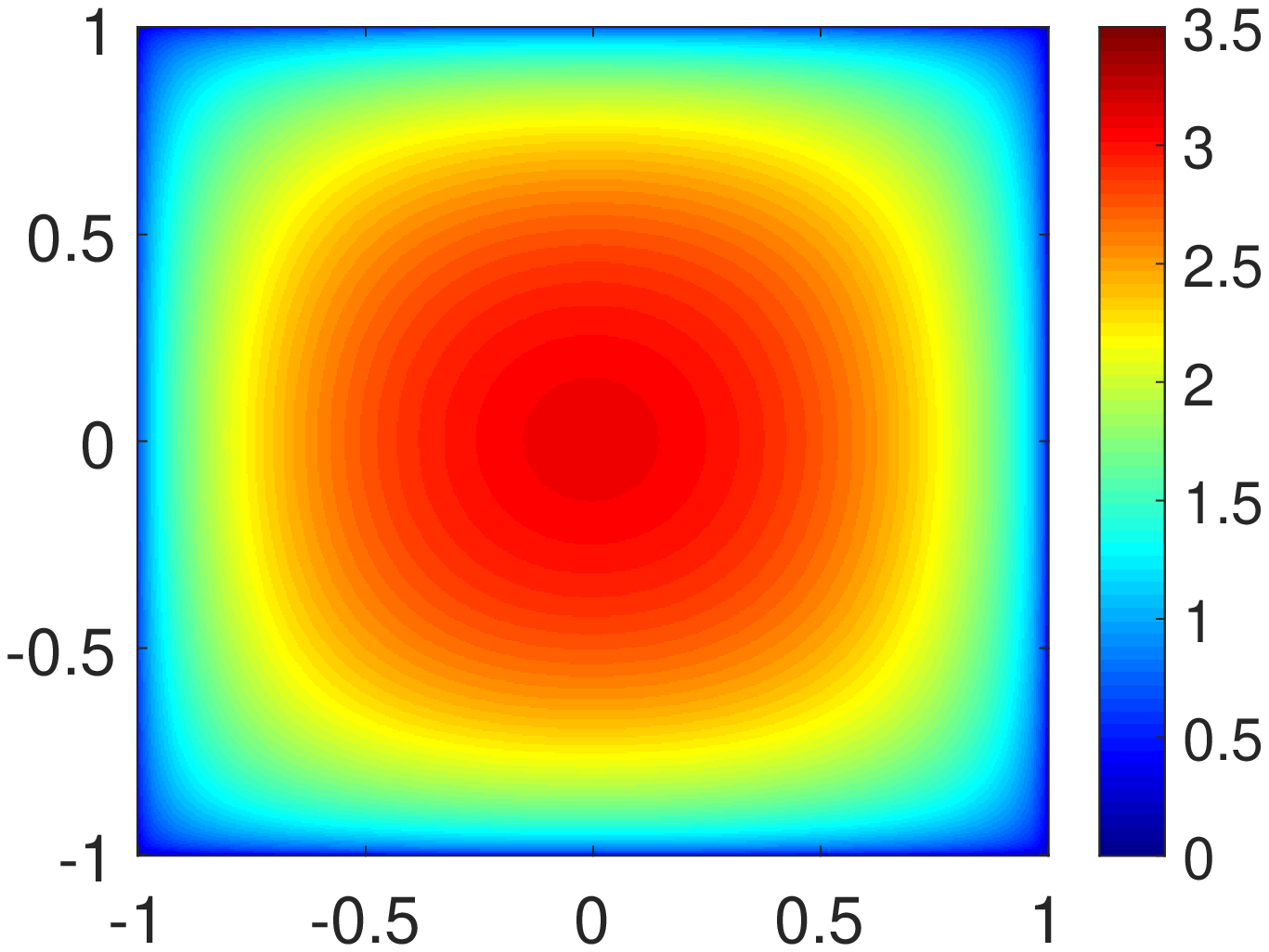}
    \end{minipage}
    }
    \subfigure[$\beta=0.8$, $\lambda=0.5$]{
    \begin{minipage}{3.5cm}
    \centering %子图居中
    \includegraphics[width=3.5cm]{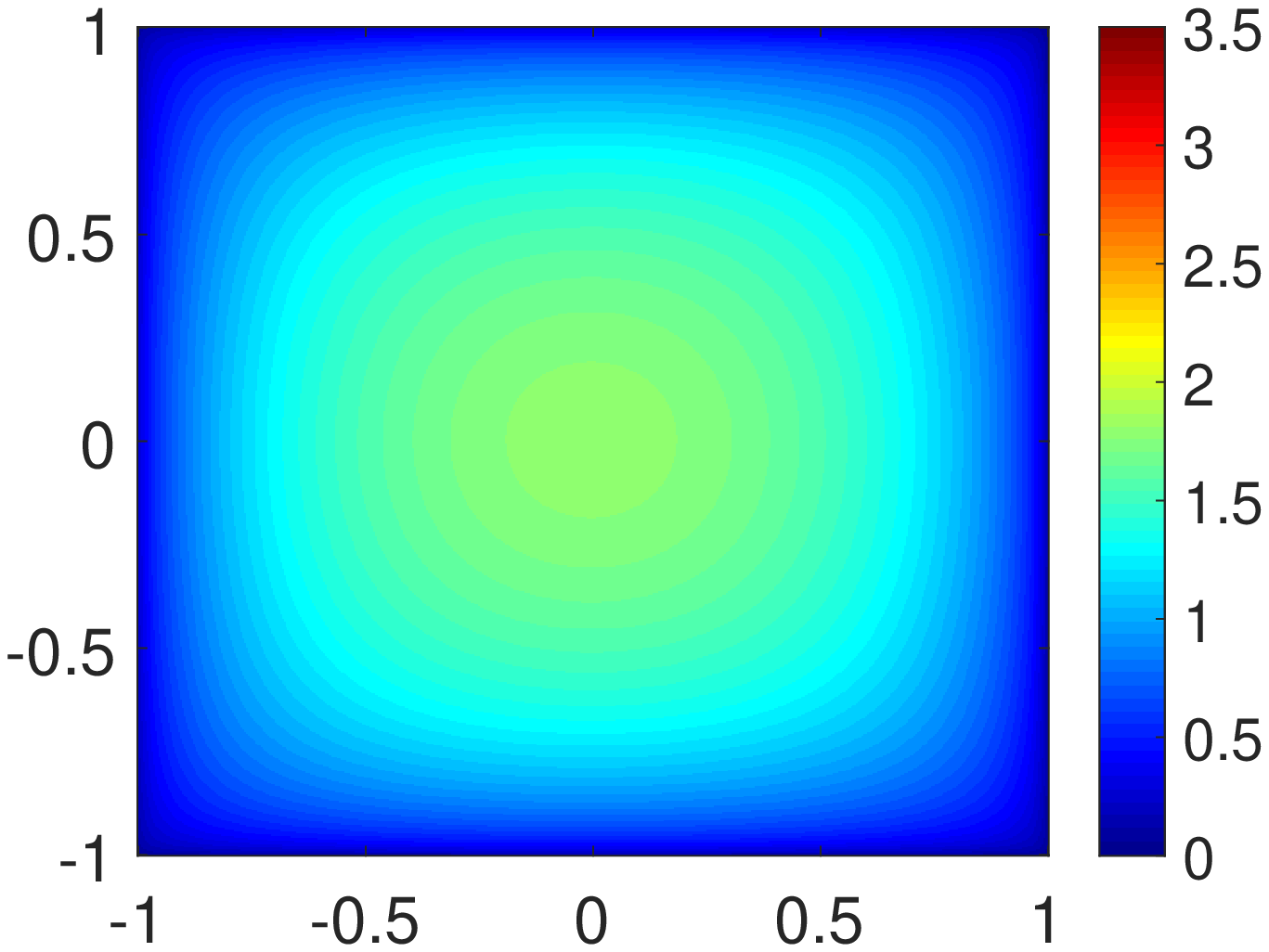}
    \end{minipage}
    }
      \subfigure[$\beta=1.2$, $\lambda=0.5$]{
    \begin{minipage}{3.5cm}
    \centering %子图居中
    \includegraphics[width=3.5cm]{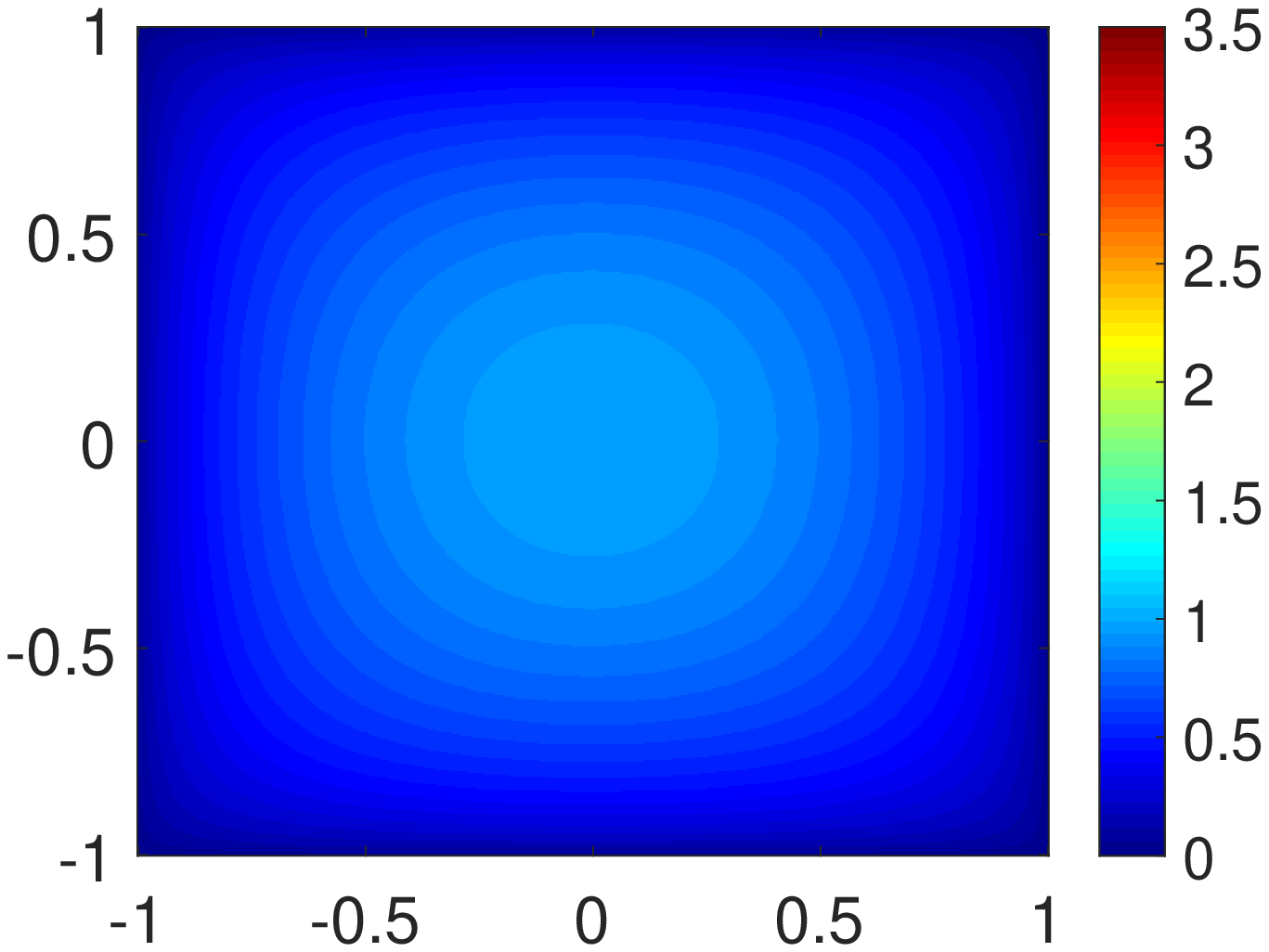}
    \end{minipage}
    }

    \subfigure[$\beta=1.5$, $\lambda=0$]{
    \begin{minipage}{3.5cm}
    \centering %子图居中
    \includegraphics[width=3.5cm]{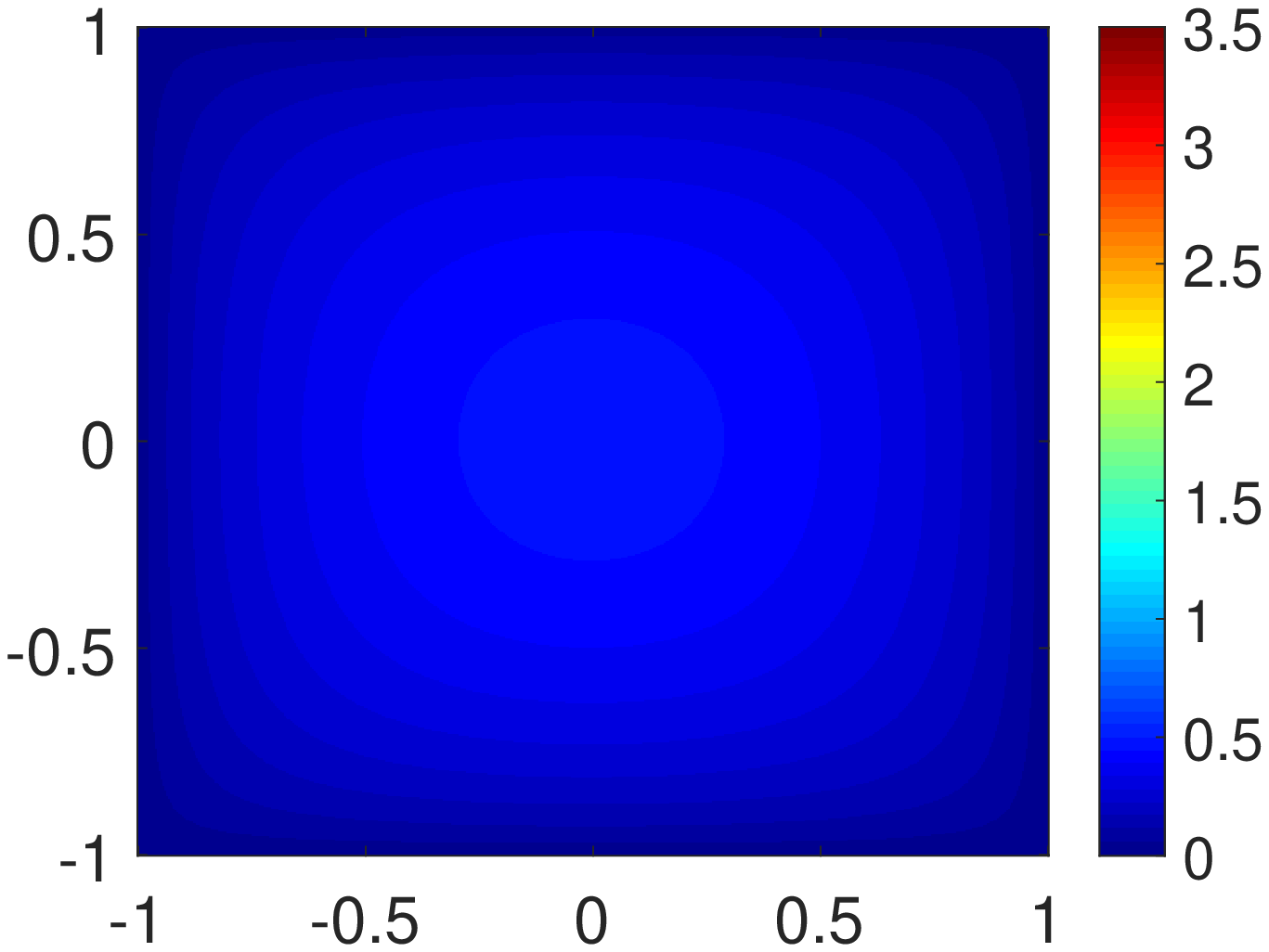}
    \end{minipage}
    }
      \subfigure[$\beta=1.5$, $\lambda=0.5$]{
    \begin{minipage}{3.5cm}
    \centering %子图居中
    \includegraphics[width=3.5cm]{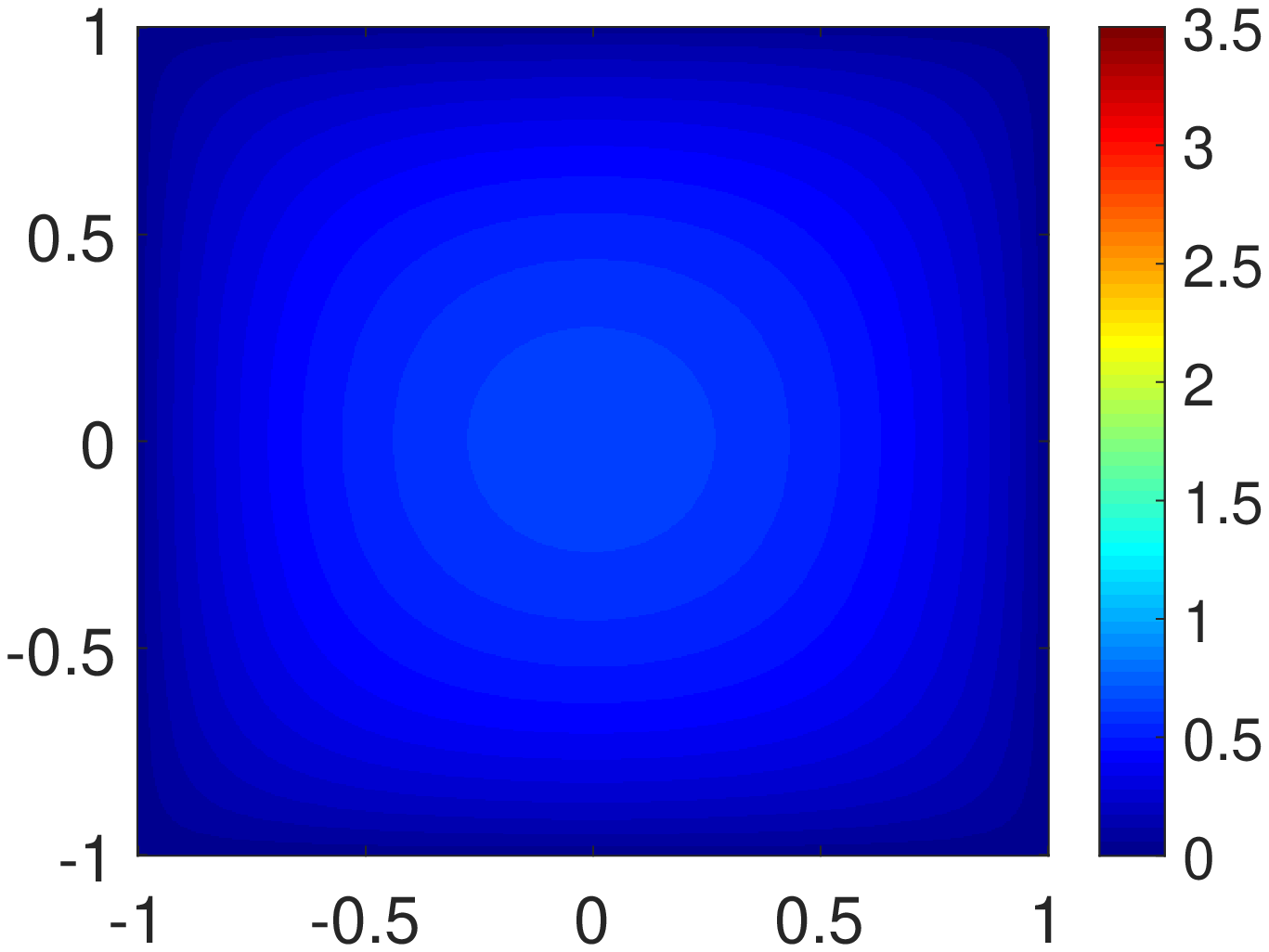}
    \end{minipage}
    }

  \centering\caption{Dependence of the mean exit time $u$ on $\beta$ and $\lambda$.}\label{figDepen}
\end{figure}
\end{example}
\section{Conclusion}

This paper provides the finite difference schemes for the two dimensional tempered fractional Laplacian, being physically introduced and mathematically defined in \cite{Deng:17}. The operator is written as the weighted integral of a weak singular function by introducing the auxiliary function $\phi_{\gamma}$. The weighted trapezoidal rule is used to approximate the integration of the weak singular part and the bilinear interpolation for the rest of the integration domain. The detailed error estimates are presented for the designed numerical schemes of the tempered fractional Poisson equation. Extensive numerical experiments are performed to verify the convergence rates and show the effectiveness of the scheme, and the quantity of mean first exit time in statistical physics is simulated. The schemes and its numerical analysis still work well for the case that $\lambda=0$, i.e., fractional Laplacian; the corresponding numerical experiments are also given.

%In this article, we perform a finite difference approximation for the tempered fractional Laplacian in two dimension. We write expression (\ref{idef1}) as the weighted integral of a weak singular function by introducing the auxiliary function $\phi_{\gamma}$; we take the weighted trapezoidal rule to approximate the integration for the weak singular and the bilinear interpolation for the rest of integration domain. Next, we give the proof of truncation error and error estimates of solving the tempered fractional Poisson equation, of which we mainly use the Taylor's expansion and the error estimates about the bilinear interpolation. In Numerical experiments, we verify the predicted convergence rates and discuss the parameters' influence on the error and convergence rates. What's more, we verify the validity of the method for fractional Laplacian by setting the parameter $\lambda=0$. In the future, we'll apply the idea of our method to other non-local diffusion equations in two dimension.

\section*{Acknowledgments} This work was supported by the National Natural Science Foundation of China under Grant No. 11671182, and the Fundamental Research Funds for the Central Universities under Grant No. lzujbky-2017-ot10.

\section*{Appendix}

\appendix

\section{Numerically calculating $(\Delta+\lambda)^{\frac{\beta}{2}}$ performed on a given function}

According to the equation $-(\Delta+\lambda)^{\frac{\beta}{2}} u(x,y)=f(x,y)$, we can compute $-(\Delta+\lambda)^{\frac{\beta}{2}} u(x,y)$ to get the source term $f(x,y)$. Since the singularity and non-locality of $-(\Delta+\lambda)^{\frac{\beta}{2}} u(x,y)$, one can't directly approximate it by the trapezoidal rule. Now we provide the technique to calculate it. For fixed $(x,y)$, we denote
\begin{equation}
  \begin{split}
   &r_1=\sup_{\begin{subarray}{c}
    (\xi,\eta)\in\partial \Omega
    \end{subarray}}\max(|x-\xi|,|y-\eta|)\\
    &r_2=\inf_{\begin{subarray}{c}
    (\xi,\eta)\in\partial \Omega
    \end{subarray}}\sqrt{(x-\xi)^2+(y-\eta)^2}.\\
  \end{split}
\end{equation}
Without loss of generality, we set $\Omega=(-1,1)\times(-1,1)$. For any $(x,y)\in\Omega$, we denote $A_1$ as a square whose length is $2r_1$ and center is located at $(x,y)$ and $A_2$ as a square whose length is $2r_2$ and center is located at $(x,y)$. To compute the source term $f(x,y)$, we divide the domain into four parts, i.e., $R\times R=(R\times R)/A_1\bigcup (A_1/\Omega)\bigcup(\Omega/A_2)\bigcup A_2$, shown in Figure \ref{Integral_region}.%, we present the regional subdivision for any fixed point.

\begin{figure}[H]
  \centering
  % Requires \usepackage{graphicx}
  \includegraphics[width=6cm,height=6cm,angle=0]{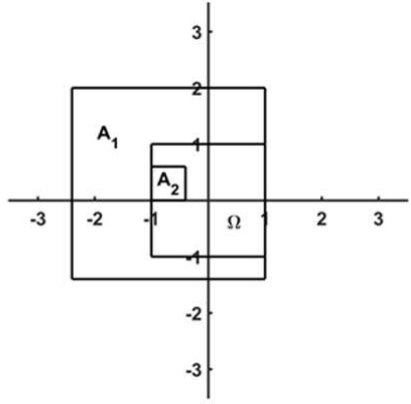}

  \centering\caption{Division of the integral region for fixed $(x,y)$}\label{Integral_region}
\end{figure}

For the term
\begin{equation}\label{equimpleR1}
\int\int_{(R\times R)/(A_1)}\frac{u(\xi,\eta)-u(x,y)}{e^{\lambda\sqrt{(x-\xi)^2+(y-\eta)^2}}\left(\sqrt{(x-\xi)^2+(y-\eta)^2}\right)^{{2+\beta}}}d\xi d\eta,
\end{equation}
since $\textbf{supp}~u(x,y)\in\Omega$, Eq. (\ref{equimpleR1}) can be rewritten as
\begin{equation}\label{equimpleR2}
-u(x,y)\int\int_{(R\times R)/(A_1)}\frac{1}{e^{\lambda\sqrt{(x-\xi)^2+(y-\eta)^2}}\left(\sqrt{(x-\xi)^2+(y-\eta)^2}\right)^{{2+\beta}}}d\xi d\eta.
\end{equation}
Next, we establish polar coordinates at $(x,y)$ and let $x-\xi=r\cos(\theta)$, $y-\eta=r\sin(\theta)$. Then, by simple calculation, we can obtain 
\begin{equation}\label{equappa1}
  \begin{split}
    &\int\int_{(R\times R)/(A_1)}\frac{1}{e^{\lambda\sqrt{(x-\xi)^2+(y-\eta)^2}}\left(\sqrt{(x-\xi)^2+(y-\eta)^2}\right)^{{2+\beta}}}d\xi d\eta\\
   =&\int_{0}^{\frac{\pi}{4}}\int_{\frac{r_1}{\cos(\theta)}}^{\infty}\frac{1}{r^{1+\beta}e^{\lambda r}}dr d\theta+\int_{\frac{\pi}{4}}^{\frac{2\pi}{4}}\int_{\frac{r_1}{\cos(\frac{\pi}{2}-\theta)}}^{\infty}\frac{1}{r^{1+\beta}e^{\lambda r}}dr d\theta\\
   &+\int_{\frac{2\pi}{4}}^{\frac{3\pi}{4}}\int_{\frac{r_1}{\cos(\theta-\frac{\pi}{2})}}^{\infty}\frac{1}{r^{1+\beta}e^{\lambda r}}dr d\theta+\int_{\frac{3\pi}{4}}^{\frac{4\pi}{4}}\int_{\frac{r_1}{\cos(\pi-\theta)}}^{\infty}\frac{1}{r^{1+\beta}e^{\lambda r}}dr d\theta\\
   &+\int_{\frac{4\pi}{4}}^{\frac{5\pi}{4}}\int_{\frac{r_1}{\cos(\theta-\pi)}}^{\infty}\frac{1}{r^{1+\beta}e^{\lambda r}}dr d\theta+\int_{\frac{5\pi}{4}}^{\frac{6\pi}{4}}\int_{\frac{r_1}{\cos(\frac{3\pi}{2}-\theta)}}^{\infty}\frac{1}{r^{1+\beta}e^{\lambda r}}dr d\theta\\
   &+\int_{\frac{6\pi}{4}}^{\frac{7\pi}{4}}\int_{\frac{r_1}{\cos(\theta-\frac{3\pi}{2})}}^{\infty}\frac{1}{r^{1+\beta}e^{\lambda r}}dr d\theta+\int_{\frac{7\pi}{4}}^{\frac{8\pi}{4}}\int_{\frac{r_1}{\cos(2\pi-\theta)}}^{\infty}\frac{1}{r^{1+\beta}e^{\lambda r}}dr d\theta\\
   =&8\int_{0}^{\frac{\pi}{4}}\int_{\frac{r_1}{\cos(\theta)}}^{\infty}\frac{1}{r^{1+\beta}e^{\lambda r}}dr d\theta.
  \end{split}
  \end{equation}
  When $\lambda=0$, we have
  \begin{equation}
  \begin{split}
   &\int\int_{(R\times R)/(A_1)}\frac{1}{e^{\lambda\sqrt{(x-\xi)^2+(y-\eta)^2}}\left(\sqrt{(x-\xi)^2+(y-\eta)^2}\right)^{{2+\beta}}}d\xi d\eta\\
   =&8\int_{0}^{\frac{\pi}{4}}\int_{\frac{r_1}{\cos(\theta)}}^{\infty}\frac{1}{r^{1+\beta}}dr d\theta\\
   =&\frac{8}{\beta}\int_{0}^{\frac{\pi}{4}}\left(\frac{r_1}{\cos(\theta)}\right)^{-\beta}d\theta.
  \end{split}
\end{equation}
We just approximate it by the trapezoidal rule in a finite interval. When $\lambda\neq 0$, we use the trapezoidal rule to approximate (\ref{equappa1}) after suitable truncation.

For the term
\begin{equation}\label{equimpleA21}
\int\int_{A_2}\frac{u(\xi,\eta)-u(x,y)}{e^{\lambda\sqrt{(x-\xi)^2+(y-\eta)^2}}\left(\sqrt{(x-\xi)^2+(y-\eta)^2}\right)^{{2+\beta}}}d\xi d\eta,
\end{equation}
using its symmetry leads to
\begin{equation}\label{equimpleA22}
\begin{split}
  &\int\int_{A_2}\frac{u(\xi,\eta)-u(x,y)}{e^{\lambda\sqrt{(x-\xi)^2+(y-\eta)^2}}\left(\sqrt{(x-\xi)^2+(y-\eta)^2}\right)^{{2+\beta}}}d\xi d\eta\\
   =&\int_{-r_2}^{r_2}\int_{-r_2}^{r_2}\frac{u(x+\xi,y+\eta)-u(x,y)}{e^{\lambda\sqrt{\xi^2+\eta^2}}\left(\sqrt{\xi^2+\eta^2}\right)^{{2+\beta}}}d\xi d\eta\\
   =&\int_{0}^{r_2}\int_{0}^{r_2}\frac{u(x+\xi,y+\eta)+u(x-\xi,y-\eta)+u(x+\xi,y-\eta)+u(x-\xi,y+\eta)-4u(x,y)}{e^{\lambda\sqrt{\xi^2+\eta^2}}\left(\sqrt{\xi^2+\eta^2}\right)^{{2+\beta}}}d\xi d\eta.\\
   \end{split}
\end{equation}

Because of the weak singularity, we try to compute it in polar coordinates. Let $\xi=r\cos(\theta)$, $\eta=r\sin(\theta)$. Then Eq. (\ref{equimpleA22}) can be rewritten as
\begin{equation}\label{equimpleA23}
 \begin{split}
 &\int\int_{A_2}\frac{u(\xi,\eta)-u(x,y)}{e^{\lambda\sqrt{(x-\xi)^2+(y-\eta)^2}}\left(\sqrt{(x-\xi)^2+(y-\eta)^2}\right)^{{2+\beta}}}d\xi d\eta\\
   =&\int_{0}^{\frac{\pi}{4}}\int_{0}^{\frac{r_2}{\cos(\theta)}}\left(u(x+r\cos(\theta),y+r\sin(\theta))+u(x-r\cos(\theta),y+r\sin(\theta))\right.\\
   &\left.+u(x+r\cos(\theta),y-r\sin(\theta))+u(x-r\cos(\theta),y-r\sin(\theta))-4u(x,y)\right)r^{-1-\beta}e^{-\lambda r}dr d\theta\\
   &+\int_{\frac{\pi}{4}}^{\frac{2\pi}{4}}\int_{0}^{\frac{r_2}{\cos(\frac{\pi}{2}-\theta)}}\left(u(x+r\cos(\theta),y+r\sin(\theta))+u(x-r\cos(\theta),y+r\sin(\theta))\right.\\
   &\left.+u(x+r\cos(\theta),y-r\sin(\theta))+u(x-r\cos(\theta),y-r\sin(\theta))-4u(x,y)\right)r^{-1-\beta}e^{-\lambda r}dr d\theta.\\
  \end{split}
\end{equation}
In (\ref{equimpleA23}), for some special function, such as $u(x,y)=(1-x^2)^2(1-y^2)^2$, we can expand it as
 \begin{equation}
   \begin{split}
     &\left(u(x+r\cos(\theta),y+r\sin(\theta))+u(x-r\cos(\theta),y+r\sin(\theta))\right.\\
   &\left.+u(x+r\cos(\theta),y-r\sin(\theta))+u(x-r\cos(\theta),y-r\sin(\theta))-4u(x,y)\right)r^{-1-\beta}e^{-\lambda r}\\
     =&4 r^{1-\beta}e^{-\lambda r} \left(r^6 \sin ^4(\theta) \cos ^4(\theta)+6 r^4 x^2 \sin ^4(\theta) \cos ^2(\theta)+6 r^4 y^2 \sin ^2(\theta) \cos^4(\theta)-2 r^4 \sin ^2(\theta) \cos ^4(\theta)\right.\\
     &-2 r^4 \sin ^4(\theta) \cos ^2(\theta)+r^2 x^4 \sin ^4(\theta)+36 r^2 x^2 y^2 \sin ^2(\theta) \cos ^2(\theta)-2 r^2 x^2 \sin ^4(\theta)\\
     &-12 r^2 x^2 \sin ^2(\theta) \cos ^2(\theta)+r^2 y^4 \cos ^4(\theta)-2 r^2 y^2 \cos ^4(\theta)-12 r^2 y^2 \sin ^2(\theta) \cos ^2(\theta)\\
     &+r^2 \sin ^4(\theta)+r^2 \cos ^4(\theta)+4 r^2 \sin ^2(\theta) \cos ^2(\theta)+6 x^4 y^2 \sin ^2(\theta)-2 x^4 \sin ^2(\theta)+6 x^2 y^4 \cos ^2(\theta)\\
     &-12 x^2 y^2 \sin ^2(\theta)-12 x^2 y^2 \cos ^2(\theta)+4 x^2 \sin ^2(\theta)+6 x^2 \cos ^2(\theta)-2 y^4 \cos ^2(\theta)+6 y^2 \sin ^2(\theta)\\
     &\left.+4 y^2 \cos ^2(\theta)-2 \sin ^2(\theta)-2 \cos ^2(\theta)\right).
   \end{split}
 \end{equation}
  When $\lambda=0$, the inner integration about $r$ can be calculated analytically, so we just need to approximate the outer integration about $\theta$ by the trapezoidal rule. When $\lambda\neq0$, we can transform the inner integration about $r$ to a nonsingular numerical integration through integration by parts.

For the another two terms, 
\begin{equation}\label{equimpleR1l}
\int\int_{(A_1/\Omega)\bigcup(\Omega/A_2)}\frac{u(\xi,\eta)-u(x,y)}{e^{\lambda\sqrt{(x-\xi)^2+(y-\eta)^2}}\left(\sqrt{(x-\xi)^2+(y-\eta)^2}\right)^{2+\beta}}d\xi d\eta,
\end{equation}
can be integrated by the trapezoidal rule directly.

\section{key points of code implementation}
When solving the tempered fractional Poisson problem with Dirichlet boundary conditions in two dimension, the computational complexity need to be carefully considered.

Firstly, since the weights $w_{i,j}$ can not be got analytically, we need to calculate it numerically. In order to get the weights (\ref{equweightoffl}), we need to calculate $G_{i,j}$, $G^\xi_{i,j}$, $G^\eta_{i,j}$ and $G^{\xi\eta}_{i,j}$. It is easy to see that $G_{i,j}$ depends on the mesh size $h$. To get $G_{i,j}$ for different $h$ conveniently, we rewrite (\ref{equdefG}) as
\begin{equation}
  G_{i,j}=\frac{1}{h^{2+\beta-\gamma}}g_{i,j}~~~~~~~~i,j\in N~{\rm and} ~(i,j)\neq(0,0),
\end{equation}
where
\begin{equation}
  g_{i,j}=\int_i^{i+1}\int_j^{j+1}\left(\sqrt{p^2+q^2}\right)^{\gamma-2-\beta}dp dq~~~~~~~i,j\in N~{\rm and} ~(i,j)\neq(0,0).
\end{equation}
We can calculate $g_{i,j}$ by the trapezoidal formula. And the same skill can be used to calculate $G^\xi_{i,j}$, $G^\eta_{i,j}$ and $G^{\xi\eta}_{i,j}$ when $(i,j)\neq(0,0)$.

Secondly, for $G_{0,0}$, we can integrate it in polar coordinates to deal with the singularity, that is
  \begin{equation}\label{equapb1}
  \begin{split}
    G_{0,0}=&\int_{0}^{\frac{\pi}{2}}\int_{0}^{h}r^{\gamma-1-\beta}dr d\theta+\int_{\xi_0}^{\xi_1}\int_{\sqrt{h^2-\xi^2}}^{\eta_1}(\xi^2+\eta^2)^{\frac{\gamma-2-\beta}{2}}d\eta d\xi\\
    =&\frac{\pi}{2(\gamma-\beta)}h^{\gamma-\beta}+\int_{\xi_0}^{\xi_1}\int_{\sqrt{h^2-\xi^2}}^{\eta_1}(\xi^2+\eta^2)^{\frac{\gamma-2-\beta}{2}}d\eta d\xi.
    \end{split}
  \end{equation}
By the way, we only need to use the trapezoidal rule to calculate the second term in (\ref{equapb1}).

Thirdly, when using the polar coordinates to calculate $G^\infty$, the integration in two dimension can be translated to a bounded integration in one dimension when $\lambda=0$. When $\lambda\neq 0$, $G^\infty$ can be calculated effectively after a suitable truncation.
Lastly, when solving the linear equation $B\mathbf{U}_h=F$, the computation costs are expensive if we solve it directly. So we use the structure of the symmetric block Toeplitz matrix with Toeplitz block, the memory requirements can be reduced from $O(N^4)$ to $O(N^2)$.
 %Since the symmetry of matrix, we use the Conjugate Gradient to avoid calculating the inverse of $B$.
And the Fast Fourier transform is used to reduce computational cost from $O(N^6)$ to $O(N^2 \log N^2)$.

\end{document}